\documentclass{amsart}

\usepackage{amsfonts}
\usepackage{amssymb}
\usepackage{epsfig, color}
\newcommand{\N}{{\mathbb N}}
\newcommand{\Z}{{\mathbb Z}}

\newcommand{\ga}{$\Gamma (G, \mathcal{X} \cup \mathcal H)$}

\newcommand{\Hl}{\{ H_\lambda \}_{\lambda \in \Lambda}}
\newcommand{\cH}{\mathcal{H}}
\newcommand{\cX}{\mathcal{X}}
\newcommand{\cW}{\mathcal{W}}

\newcommand{\R}{{\mathbb R}}
\newcommand{\cc}{\rm ($2$CC)}
\newcommand{\ncc}{\rm ({\it n}CC)}
\newcommand{\lab}{{\rm Lab}}
\renewcommand{\L}{{\rm L}}

\theoremstyle{plain}
\newtheorem{lemma}{Lemma}[section]
\newtheorem{cor}[lemma]{Corollary}
\newtheorem{thm}[lemma]{Theorem}

\theoremstyle{definition}
\newtheorem*{df}{Definition}
\newtheorem{example}{Example}
\newtheorem{quest}{Question}

\theoremstyle{remark}
\newtheorem{remark}[lemma]{Remark}

\makeatletter
\@addtoreset{equation}{section} 
\makeatother

\begin{document}
\title[Groups with finitely many conjugacy classes]{Groups with finitely many conjugacy classes and their automorphisms}

\author{Ashot Minasyan}
\address{
School of Mathematics,
University of Southampton, Highfield, Southampton, SO17 1BJ, United
Kingdom.}
\email{aminasyan@gmail.com}

\begin{abstract} We combine classical methods of combinatorial group theory with the theory of small
cancellation over relatively hyperbolic groups to construct finitely generated torsion-free groups that
have only finitely many classes of conjugate elements. Moreover, we present several results concerning
embeddings into such groups.

As another application of these techniques, we prove that every countable group $C$ can be realized as a group of outer automorphisms
of a group $N$, where $N$ is a finitely generated group having Kazhdan's property (T)
and containing exactly two conjugacy classes.
\end{abstract}
\thanks{This work was supported by the Swiss
National Science Foundation Grant $\sharp$~PP002-68627.}
\keywords{Conjugacy Classes, Relatively Hyperbolic Groups, Outer Automorphism Groups}
\subjclass[2000]{20F65, 20E45, 20F28}

\maketitle

\section{Introduction}
We shall start with
\begin{df} Suppose that $n \ge 2$ is an integer. We will say that a group $M$ has the property {\ncc} if
there are exactly $n$ conjugacy classes of elements in $M$.
\end{df}

Note that a group $M$ has {\cc} if and only if any two non-trivial elements are conjugate in $M$.
For two elements $x,y $ of some group $G$, we shall write $x \stackrel{G}{\sim} y$ if $x$ and $y$ are conjugate in $G$,
and $x \stackrel{G}{\nsim} y$ if they are not.

For a group $G$, denote by $\pi(G)$ the set of all finite orders of elements of $G$.
A classical theorem of G. Higman, B. Neumann and H. Neumann (\cite{HNN}) states that every countable group $G$ can be
embedded into a countable (but infinitely generated) group $M$, where any two elements of the same order are conjugate and
$\pi(M)=\pi(G)$.

For any integer $n\ge 2$, take $G=\Z/2^{n-2}\Z$ and embed $G$ into a countable group $M$ according to the theorem above.
Then $card(\pi(M))=card(\pi(G))=n-1$. Since, in addition, $M$ will always contain an element of infinite order, the theorem of
Higman-Neumann-Neumann implies that $G$ has {\ncc}.

Another way to construct infinite groups with finitely many conjugacy classes was suggested by S. Ivanov \cite[Thm. 41.2]{Olsh0}, who showed
for every sufficiently large prime $p$ there is an infinite $2$-generated group $M_p$ of exponent $p$ possessing exactly $p$ conjugacy classes.
The group $M_p$ is constructed as a direct limit of word hyperbolic groups, and, as noted in \cite{Osin-SCT}, it is impossible to
obtain an infinite group with {\cc} in the same manner.

In the recent paper \cite{Osin-SCT} D. Osin developed a theory of small cancellation over relatively hyperbolic groups and
used it to obtain the following remarkable result:

\begin{thm}[\cite{Osin-SCT}, Thm. 1.1]\label{thm:osin's-main} Any countable group $G$ can be embedded into a
$2$-generated group $M$ such that any two elements of the same
order are conjugate in $M$ and $\pi(M)=\pi(G)$.
\end{thm}

Applying this theorem to the group $G=\Z/2^{n-2}\Z$ one can show that for each integer $n\ge 2$ there exists a $2$-generated group
with ${\ncc}$. And when $n=2$ we get a $2$-generated torsion-free group that has exactly two conjugacy classes.

The presence of elements of finite orders in the above constructions was important, because if two elements have different orders,
they can never be conjugate. So, naturally, one can ask the following

\begin{quest} \label{q:1} Do there exist
torsion-free (finitely generated) groups with {\ncc}, for any integer $n \ge 3$?
\end{quest}

Note that if $G$ is the finitely generated group with {\cc} constructed by Osin, then the $m$-th direct power $G^m$ of $G$
is also a finitely generated torsion-free group which satisfies ($2^m$CC). But what if we want to achieve a torsion-free
group with ($3$CC)? With this purpose one could come up with

\begin{quest} \label{q:2} Suppose that $G$ is a countable torsion-free group and $x,y \in  G$ are non-conjugate. Is it possible to
embed $G$ into a group $M$, which has ($3$CC), so that $x$ and $y$ stay non-conjugate in $M$?
\end{quest}

Unfortunately, the answer to Question \ref{q:2} is negative as the following example shows.

\begin{example}\label{ex:Klein_bottle_gp}
Consider the group \begin{equation} \label{eq:Kl-bot} G_1=\langle a,t~\|~tat^{-1}=a^{-1}\rangle \end{equation}
which is isomorphic to the non-trivial semidirect product
$\Z \rtimes \Z$. Note that $G_1$ is torsion-free, and $t$ is not conjugated to $t^{-1}$ in $G_1$ because $t \nsim t^{-1}$ in the infinite
cyclic group $\langle t \rangle$ which is canonically isomorphic to the quotient of $G_1$ by the normal closure
of $a$. However, if $G_1$ is embedded into a (3CC)-group $M$, it is easy to see that every element of $M$ will be conjugated to its inverse
(indeed, if $y \in M \setminus \{1\}$ and $y \stackrel{M}{\nsim} y^{-1}$ then $y^\epsilon \stackrel{M}{\sim} a \stackrel{M}{\sim} a^{-1}$,
 for some $\epsilon \in \{1,-1\}$, hence $y^\epsilon \stackrel{M}{\sim} y^{-\epsilon}$ -- a contradiction).
In particular, $t \stackrel{M}{\sim} t^{-1}$.
\end{example}

An analog of the above example can be given for each $n \ge 3$ -- see Section \ref{sec:groups_with_fCC}. This example shows that, in order
to get a positive result, one would have to strengthen the assumptions of Question \ref{q:2}.

Let $G$ be a group. Two elements $x,y \in G$  are said to be {\it commensurable} if there exist
$k,l \in \Z\setminus \{0\}$ such that $x^k$ is conjugate to $y^l$. We will use the notation $x \stackrel{G}{\approx} y$ if $x$
and $y$ are commensurable in $G$. In the case when $x$ is not commensurable with $y$ we will write
$x \stackrel{G}{\not\approx} y$.
Observe that commensurability, as well as conjugacy, defines an equivalence relation on the set of elements of $G$.
It is somewhat surprising that if one replaces the words "non-conjugate" with the words  "non-commensurable" in Question \ref{q:2},
the answer becomes positive:

\begin{cor} \label{cor:non-comm_sep_emb} Assume that $G$ is a countable torsion-free group, $n \in \N$, $n \ge 2$, and
$x_1,\dots,x_{n-1} \in G\setminus \{1\}$ are pairwise non-commensurable. Then there exists a group $M$ and an
injective homomorphism $\varphi:G \to M$ such that
\begin{itemize}
\item[1.] $M$ is torsion-free and generated by two elements;
\item[2.] $M$ has {\ncc};
\item[3.] $M$ is $2$-boundedly simple;
\item[4.] the elements $\varphi(x_1),\dots,\varphi(x_{n-1})$ are pairwise non-commensurable in $M$.
\end{itemize}
\end{cor}

Recall that a group $G$ is said to be $k$-{\it boundedly simple} if for any $x,y \in G\setminus \{1\}$ there exist
$l \le k$ and $g_1,\dots,g_l \in G$ such that $x = g_1yg_1^{-1}\cdots g_lyg_l^{-1}$ in $G$. A group is called
{\it boundedly simple} if it is $k$-boundedly simple for some $k \in \N$. Evidently every boundedly simple group is simple; the converse
is not true in general. For example, the infinite alternating group $A_\infty$ is simple but not boundedly simple because conjugation
preserves the type of the decomposition of a permutation into a product of cycles.
First examples of torsion-free finitely generated boundedly simple groups were constructed by A. Muranov (see \cite[Thm. 2]{Mur1},
\cite[Thm. 1]{Mur3}).

Corollary \ref{cor:non-comm_sep_emb} is an immediate consequence of a more general Theorem \ref{thm:cc-emb} that will be proved in Section
\ref{sec:groups_with_fCC}.

Applying Corollary \ref{cor:non-comm_sep_emb} to the group $G=F(x_1,\dots,x_{n-1})$, which is free on the set $\{x_1,\dots,x_{n-1}\}$,
and its non-commensurable elements $x_1,\dots,x_{n-1}$, we obtain
a positive answer to Question \ref{q:1}:

\begin{cor} \label{cor:existence} For every integer $n \ge 3$ there exists a torsion-free
$2$-boundedly simple group satisfying {\ncc} and generated by two elements.
\end{cor}

(In the case when $n=2$ the above statement was obtained by Osin in \cite[Cor. 1.3]{Osin-SCT}.)
In fact, for any (finitely generated) torsion-free group $H$ we can set
$G=H*F(x_1,\dots,x_{n-1})$, and then use Corollary \ref{cor:non-comm_sep_emb} to embed
$G$ into a group $M$ enjoying the properties $1-4$ from its claim. Since there is a continuum of pairwise
non-isomorphic $2$-generated torsion-free
groups (\cite{Camm}), and a finitely generated group can contain at most countably
many of different $2$-generated subgroups, this shows that there must be continually many pairwise non-isomorphic
groups satisfying properties
$1-3$ from Corollary \ref{cor:non-comm_sep_emb}.

Recall that the rank $rank(G)$ of a group $G$ is the minimal number of elements required to generate $G$.
In Section \ref{sec:norm_sbgps} we show how classical theory of HNN-extensions allows to construct different
embeddings into (infinitely generated) groups that have finitely many classes of
conjugate elements, and in Section \ref{sec:add_fg} we use Osin's results (from \cite{Osin-SCT}) regarding quotients of
relatively hyperbolic groups to prove

\begin{thm} \label{thm:emb-ext-non-simple} Let $H$ be a torsion-free countable group and let $M \lhd H$ be a non-trivial normal subgroup.
Then $H$ can be isomorphically embedded into a torsion-free group $Q$, possessing a normal subgroup $N \lhd Q$, such that
\begin{itemize}
\item $Q=H\cdot N$ and $H \cap N=M$ (hence $Q/N \cong H/M$);
\item $N$ has {\cc};
\item $\forall~x,y \in Q\setminus \{1\}$, $x \stackrel{Q}{\sim} y$ if and only if $\varphi(x) \stackrel{Q/N}{\sim} \varphi(y)$, where
$\varphi: Q \to Q/N$ is the natural homomorphism; 
\item $rank(N)=2$ and $rank(Q) \le rank(H/M)+2$.
\end{itemize}
\end{thm}

This theorem implies that if $Q/N \cong H/M$ has exactly $(n-1)$ conjugacy classes (e.g., if it is finite), then the group $Q$ will have {\ncc}
and will not be simple (if $n \ge 3$). Thus it may be used to build {\ncc}-groups in a recursive manner.
It also allows to obtain embeddings of countable torsion-free
groups into {\ncc}-groups, which we could not get by using Corollary \ref{cor:non-comm_sep_emb}. For instance,
as we saw in Example \ref{ex:Klein_bottle_gp}, the fundamental group of the Klein bottle
$G_1$, given by \eqref{eq:Kl-bot}, can not be embedded into a ($3$CC)-group $M$ so that
$t \stackrel{M}{\nsim} t^{-1}$. However, with $4$ conjugacy classes this is already possible: see Corollary \ref{cor:klein-bottle-emb}
in Section \ref{sec:add_fg}. The idea is as follows: the group $G_1$ can be mapped onto $\Z/3\Z$ in such a way
that the images of the elements $t$ and $t^{-1}$ are distinct. Let $M$ be the kernel of this homomorphism.
One can apply Theorem \ref{thm:emb-ext-non-simple} to the pair $(G_1,M)$ to obtain the
required embedding of $G_1$ into a group $Q$. And since $\Z/3\Z$ has exactly $3$ conjugacy classes, the group $Q$ will have
($4$CC).

An application of Theorem \ref{thm:emb-ext-non-simple} to the case when $H=\Z$ and $M=2\Z \lhd H$ also provides an affirmative answer to a
question of A. Izosov from \cite[Q. 11.42]{Kourovka}, asking whether there exists a torsion-free ($3$CC)-group
$Q$ that contains a normal subgroup $N$ of index $2$.

The goal of the second part of this article is to show that every countable group can be realized as a group of
outer automorphisms of some finitely generated {\cc}-group. This problem has some historical background: in \cite{Matumoto}
T. Matumoto proved that every group is a group of outer automorphisms of some group (in contrast, there are groups, e.g., $\Z$,
that are not full automorphism groups of any group); M. Droste, M. Giraudet, R. G\"{o}bel (\cite{DGG}) showed that
for every group $C$ there exists a simple group $S$ such that $Out(S) \cong C$; I. Bumagina and D. Wise in \cite{Bum-Wise}
proved that each countable group $C$ is isomorphic to $Out(N)$ where $N$ is a $2$-generated subgroup of a countable $C'(1/6)$-group,
and if, in addition, $C$ is finitely presented then one can choose $N$ to be residually finite.

In Section \ref{sec:comb_paths} we establish a few useful statements regarding paths in the Cayley
graph of a relatively hyperbolic group $G$, and apply them in Section \ref{sec:smal_canc} to obtain
small cancellation quotients of $G$ satisfying certain conditions. Finally, in Section \ref{sec:every_gp=out}
we prove the following

\begin{thm}\label{thm:gp=out} Let $C$ be an arbitrary countable group. Then for every non-elemen\-tary torsion-free
word hyperbolic group $F_1$ there exists a torsion-free group $N$ satisfying the following properties:
\begin{itemize}
\item $N$ is a $2$-generated quotient of $F_1$;
\item $N$ has {\cc};
\item $Out(N) \cong C$.
\end{itemize}
\end{thm}

The principal difference between this theorem and the result of \cite{Bum-Wise} is that our group $N$ is
torsion-free and simple. Moreover, if one applies Theorem \ref{thm:gp=out} to the case when $F_1$ is a torsion-free
hyperbolic group with Kazhdan's property (T) (and recalls that every quotient of a group with property (T) also has (T)), one will
get

\begin{cor} \label{cor:prop_T} For any countable group $C$ there is a $2$-generated group $N$ such that $N$ has
{\cc} and Kazhdan's property (T), and $Out(N) \cong C$.
\end{cor}

The reason why Kazhdan's property (T) is interesting in this context is the question from \cite[p. 134]{Harpe-Valette} which asked
whether there exist groups that satisfy property (T) and have infinite outer automorphism groups (it can be motivated by
a theorem of F. Paulin \cite{Paulin} which claims that the outer automorphism group is finite for any word hyperbolic group with property (T)).
Positive answers to this question
were obtained (using different methods) by Y. Ollivier and D. Wise \cite{Ol-Wise}, Y. de Cornulier \cite{de_Cornul}, and
I. Belegradek and D. Osin \cite{Bel-Osin}.
Corollary \ref{cor:prop_T} not only shows that the group of outer automorphisms of a group $N$ with property (T)  can be infinite, but also
demonstrates that there are no restrictions whatsoever on $Out(N)$.

{\bf Acknowledgements.} The author would like to thank D. Osin for fruitful discussions and encouragement.
%
%


\section{Relatively hyperbolic groups}
Assume that $G$ is a group, $\Hl$ is a
fixed collection of subgroups of $G$ (called {\it peripheral
subgroups}), and $\cX$ is a subset of $G$. The subset  $\cX$ is called a {\it relative
generating set of $G$} with respect to $\Hl $ if $G$ is generated by
$\cX \cup \bigcup_{\lambda \in \Lambda} H_\lambda $. In this case $G$  a quotient of the free product
$$F=\left( \ast _{\lambda\in \Lambda } H_\lambda  \right) \ast F(\cX),$$
where $F(\cX)$ is the free group with basis $\cX$.
Let $\mathcal R $ be a subset of $F$ such that the kernel
of the natural  epimorphism $F\to G$ is the normal closure of $\mathcal R $ in the group $F$;
then we will say that $G$ has {\it relative presentation}
\begin{equation}\label{eq:pres_of_G}
\langle \cX,\; \{H_\lambda\}_{\lambda\in \Lambda}~ \|~ R=1,\, R\in\mathcal R\rangle .
\end{equation}
If the sets $\cX$ and $\mathcal R$ are finite, the relative
presentation (\ref{eq:pres_of_G}) is said to be {\it finite}.

Set $\mathcal H=\bigsqcup_{\lambda\in \Lambda} (H_\lambda\setminus \{ 1\} )$.
A finite relative presentation \eqref{eq:pres_of_G} is said to satisfy a {\it linear
relative isoperimetric inequality} if there exists $C>0$ such that, for
every word $w$ in the alphabet $\cX\cup \mathcal{H}$ (for convenience,  we will further assume that
$\cX^{-1}=\cX$) representing the identity in the group
$G$, one has
$$w\stackrel{F}{=}\prod\limits_{i=1}^k f_i^{-1}R_i^{\pm 1}f_i,$$
with equality in the group $F$, where $R_i\in \mathcal{R}$,
$f_i\in F $, for $i=1, \ldots , k$, and  $k\le C\| w\| $,
where $\| w\|$ is the length of the word $w$.

The next definition is due to Osin (see \cite{Osin-RHG}):

\begin{df} the group $G$ is called {\it  hyperbolic relative to} (the
collection of peripheral subgroups) $\Hl $, if $G$ admits a finite relative
presentation (\ref{eq:pres_of_G})  satisfying a  linear
relative isoperimetric inequality.
\end{df}

This definition is
independent of the choice of the finite generating set $\cX$ and
the finite set $\mathcal R$ in (\ref{eq:pres_of_G}) (see \cite{Osin-RHG}).
We would  also like to note that, in general, it does not require the group $G$ to be finitely generated, which
will be important in this paper.
The definition immediately implies the following basic facts:

\begin{remark}[\cite{Osin-RHG}] \label{rem:free_prod_rel_hyp}
(a) Let $\{H_\lambda\}_{\lambda \in \Lambda}$ be an arbitrary family of groups. Then the free product
$G=\ast_{\lambda\in \Lambda} H_\lambda$ will be hyperbolic relative to $\Hl$.

(b) Any word hyperbolic group (in the sense of Gromov) is hyperbolic relative to the family $\{\{1\}\}$, where
$\{1\}$ denotes the trivial subgroup.
\end{remark}

Recall that a group $H$ is called {\it elementary} if it has a cyclic subgroup of finite index.
Further in this section we will assume that $G$ is a non-elementary group hyperbolic relative to a family of proper subgroups
$\{H_\lambda\}_{\lambda \in \Lambda}$.

An element $g \in G$ is said to be {\it parabolic} if it is conjugated to an element of $H_\lambda$ for some $\lambda \in \Lambda$.
Otherwise $g$ is said to be {\it hyperbolic}. Given a subgroup $S\le G$, we denote by $S^0$ the set of all
hyperbolic elements of $S$ of infinite order.

%
%
%

\begin{lemma}[\cite{Osin-ESBG}, Thm. 4.3, Cor. 1.7]\label{lem:Eg}
For every $g \in G^0$ the following conditions hold.
\begin{itemize}
\item[1)] The element $g$ is contained in a unique maximal elementary
subgroup $E_G(g)$ of $G$, where
\begin{equation} \label{eq:elem}
E_G(g)=\{ f\in G\; :\;
fg^nf^{-1}=g^{\pm n}\; {\rm for \; some\; } n\in \mathbb N\}.
\end{equation}
\item[2)] The group $G$ is hyperbolic relative to the collection
$\Hl\cup \{ E_G(g)\} $.
\end{itemize}
\end{lemma}

%
%
%

Recall that a non-trivial subgroup $H \le G$ is called {\it malnormal} if for every $g \in G\setminus H$,
$H \cap gHg^{-1}=\{1\}$.
The next lemma is a special case of Theorem 1.4 from  \cite{Osin-RHG}:

\begin{lemma}\label{lem:malnorm}
For any $\lambda \in \Lambda $ and any $g\notin
H_\lambda $, the intersection $H_\lambda \cap g H_\lambda g^{-1}$ is
finite. If $h \in G$, $\mu \in \Lambda$ and $\mu \neq \lambda$, then the intersection  $H_\lambda \cap hH_\mu h^{-1}$ is finite.
In particular, if $G$ is torsion-free then $H_\lambda$ is malnormal (provided that $H_\lambda \neq \{1\}$).
\end{lemma}

\begin{lemma}[\cite{Osin-RHG}, Thm. 2.40]\label{lem:exhyp}
Suppose that a group $G$ is hyperbolic relative to a collection of
subgroups $\Hl \cup \{ S_1, \ldots , S_m\} $, where $S_1, \ldots, S_m $
are word hyperbolic (in the ordinary non-relative sense). Then $G$ is hyperbolic relative to $\Hl $.
\end{lemma}

%

\begin{lemma}[\cite{Osin-RDF}, Cor. 1.4]\label{lem:HNN-rel_hyp}
Let $G$ be a group which is hyperbolic relative to a collection of
subgroups $\Hl \cup \{K \}$. Suppose that $K$ is finitely generated and
there is a monomorphism $\alpha \colon K\to H_\nu$ for some $\nu \in \Lambda$.
Then the HNN-extension $\langle G, t~\|~ txt^{-1}=\alpha (x), \, x\in K\rangle$ is hyperbolic
with respect to $\Hl$.
\end{lemma}

In \cite{Osin-SCT} Osin introduced the following notion: a subgroup $S \le G$ is {\it suitable}
if there exist two elements $g_1,g_2 \in S^0$ such that
$g_1 \stackrel{G}{\not\approx} g_2$ and $E_G(g_1) \cap E_G(g_2)=\{1\}$.

For any $S \le G$ with $S^0 \neq \emptyset$,
one sets
\begin{equation} \label{eq:E_G}
E_G(S)=\bigcap_{g \in S^0} E_G(g)
\end{equation}
which is obviously a subgroup of $G$ normalized by $S$.
Note that $E_G(S)=\{1\}$ if the subgroup $S$
is suitable in $G$.
As shown in \cite[Lemma 3.3]{SQ}, if $S$ is non-elementary and $S^0 \neq \emptyset$ then $E_G(S)$ is the unique
maximal finite subgroup of $G$ normalized by $S$.

\begin{lemma} \label{lem:hyp_suit_sbgp_free_prod} Let $\{H\}_{\lambda \in \Lambda}$ be a family of groups and
let $F$ be a torsion-free non-elementary word hyperbolic group.
Then the free product $G=(*_{\lambda \in \Lambda} H_\lambda ) *F$ is hyperbolic relative to $\Hl$ and $F$ is a suitable subgroup of $G$.
\end{lemma}

\begin{proof}
Indeed, $G$ is hyperbolic relative to $\Hl$ by Remark \ref{rem:free_prod_rel_hyp} and Lemma \ref{lem:exhyp}.
Since $F$ is non-elementary, there are elements of infinite order $x,y \in F$ such that $x \stackrel{F}{\not\approx} y$
(see, for example, \cite[Lemma 3.2]{Olsh2}). Evidently, $x$ and $y$ are hyperbolic elements of $G$ that are not commensurable
with each other, and the subgroups $E_G(x)=E_F(x) \le F$, $E_G(y)=E_F(y) \le F$ are cyclic (as elementary subgroups of a torsion-free group).
Hence $E_G(x) \cap E_G(y)=\{1\}$, and thus $F$ is suitable in $G$.
\end{proof}

\begin{lemma}[\cite{Osin-SCT}, Lemma 2.3] \label{lem:suit-inf_many_non-comm} Suppose that $G$ is a group
hyperbolic relative to a family of subgroups
$\{H_\lambda\}_{\lambda \in \Lambda}$ and $S\le G$ is a suitable subgroup. Then one can find infinitely many pairwise non-commensurable
(in $G$) elements $g_1,g_2, \dots \in S^0$ such that $E_G(g_i) \cap E_G(g_j)=\{1\}$ for all $i\neq j$.
\end{lemma}

The following theorem was proved by Osin in \cite{Osin-SCT} using the theory of small cancellation over relatively hyperbolic
groups, and represents our main tool for obtaining new quotients of such groups having
a number of prescribed properties:

\begin{thm}[\cite{Osin-SCT}, Thm. 2.4]\label{thm:main_SCT}
Let $G$ be a torsion-free group hyperbolic relative to a collection of subgroups
$\Hl $, let $S$ be a suitable subgroup of $G$, and let $T, U$ be arbitrary finite subsets of
$G$. Then there exist a group $G_1$ and an epimorphism $\eta \colon G\to G_1$ such
that:
\begin{itemize}
\item[(i)] The restriction
of $\eta$ to  $\bigcup_{\lambda \in \Lambda} H_\lambda \cup U$ is injective, 
and the group $G_1$ is hyperbolic relative to the
collection $\{\eta (H_\lambda )\}_{\lambda \in \Lambda }$;

\item[(ii)] for every $t\in T$, we have $\eta (t)\in \eta (S)$;

\item[(iii)] $\eta(S)$ is a suitable subgroup of $G_1$;

\item[(iv)] $G_1$ is torsion-free;

\item[(v)] the kernel $\ker(\eta)$ of $\eta$ is generated (as a normal subgroup of $G$) by a finite collection
of elements belonging to $T\cdot S$.
\end{itemize}
\end{thm}

We have slightly changed the original formulation of the above theorem from \cite{Osin-SCT}, demanding the injectivity on
$V=\bigcup_{\lambda \in \Lambda} H_\lambda \cup U$ (instead of just $\bigcup_{\lambda \in \Lambda} H_\lambda$) and
adding the last point concerning the generators of the kernel. The latter follows from the explicit form of the relations, imposed
on $G$ (see the proof of Thm. 2.4 in \cite{Osin-SCT}), and the former -- from part 2 of Lemma 5.1 
in \cite{Osin-SCT} and the fact that any element from $V$ has length (in the alphabet $\cX \cup \cH$) at most $N$, where
$N= \max\{|h|_{\cX \cup \cH}~:~h \in U\}+1$.

\section{Groups with finitely many conjugacy classes}\label{sec:groups_with_fCC}

\begin{lemma}\label{lem:HNN-conj} Let $G$ be a group and let $x_1,x_2,x_3,x_4 \in G$ be elements of infinite order such that
$x_1 \stackrel{G}{\not\approx} x_i$, $i=2,3,4$. Let $H=\langle G,t~\|~tx_3t^{-1}=x_4 \rangle$ be the HNN-extension of $G$ with associated
cyclic subgroups generated by $x_3$ and $x_4$. Then $x_1 \stackrel{H}{\not\approx} x_2$.
\end{lemma}

\begin{proof} Arguing by contradiction, assume that $h x_1^l h^{-1} x_2^{m}=1$ for some $h \in H$, $l,m \in \Z\setminus \{0\}$.
The element $h$ has a
reduced presentation of the form $$h=g_0t^{\epsilon_1}g_1t^{\epsilon_2}\dots t^{\epsilon_k}g_k$$ where
$g_0,\dots,g_k \in G$, $\epsilon_1,\dots, \epsilon_k \in \Z \setminus\{0\}$, and
$$\left\{ \begin{array}{rcl} g_j \notin \langle x_3 \rangle & \mbox{ if } & 1\le j \le k-1 \mbox{ and } \epsilon_j>0, \epsilon_{j+1}<0 \\
g_j \notin \langle x_4 \rangle & \mbox{ if } & 1\le j \le k-1 \mbox{ and } \epsilon_j<0, \epsilon_{j+1}>0
\end{array}\right. .$$
By the assumptions, $x_1 \stackrel{G}{\not \approx} x_2$ hence $k\ge 1$, and in the group $H$ we have
\begin{equation} \label{eq:eq_to_1} h x_1^l h^{-1} x_2^{m}=g_0t^{\epsilon_1}g_1t^{\epsilon_2}\dots t^{\epsilon_k}g_k x_1^l
g_k^{-1}t^{-\epsilon_k}\dots t^{-\epsilon_2} g_1^{-1} t^{-\epsilon_1}\tilde g_0=1,\end{equation}
where $\tilde g_0=g_0^{-1}x_2^m \in G$. By Britton's Lemma (see \cite[IV.2]{L-S}), the left hand side in \eqref{eq:eq_to_1}
can not be reduced, and this can happen only if
$g_k x_1^l g_k^{-1}$  belongs to either  $\langle x_3 \rangle$ or $\langle x_4 \rangle$ in $G$, which would contradict the assumptions.
Thus the lemma is proved.
\end{proof}

\begin{df} Suppose that $G$ is a group and $X_i \subset G$, $i \in I$, is a family of subsets.
We shall say that $X_i$, $i \in I$, are {\it independent} if no element of $X_i$ is commensurable with an element of
$X_j$ whenever $i \neq j$, $i,j \in I$.
\end{df}

\begin{lemma} \label{lem:HNN-emb} Assume that $G$ is a countable torsion-free group, $n\in \N$, $n \ge 2$,
and non-empty subsets $X_i \subset G \setminus \{1\}$, $i=1,\dots, n-1$, are independent in $G$.
Then $G$ can be (isomorphically) embedded into a countable torsion-free group $M$ in such a way that $M$ has {\ncc} and
the subsets $X_i$, $i=1,\dots,n-1$, remain independent in $M$.
\end{lemma}


\begin{proof}
For each $i=1,\dots,n-1$, fix an element $x_i \in X_i$.
First we embed $G$ into a countable torsion-free group
$G_1$ such that for each non-trivial element $g \in G$ there exist $j\in \{1,\dots,n-1\}$ and $t \in G_1$ satisfying
$t g t^{-1}=x_j$ in $G_1$, and the subsets $X_i$, $i=1,\dots,n-1$, stay independent in $G_1$.

Let $g_1,g_2,\dots$ be an enumeration of all non-trivial elements of $G$. Set $G(0)=G$ and suppose that we have already constructed
the group $G(k)$, containing $G$, so that for each $l \in \{1,\dots,k\}$ there is $j \in \{1,\dots,n-1\}$
such that the element $g_l$ is conjugated in $G(k)$ to $x_j$, and $X_i,i=1,\dots,n-1$, are independent in $G(k)$.

Suppose, at first, that $g_{k+1}$ is commensurable in $G(k)$ with an element of $X_j$ for some $j$.
Then $g_{k+1} \stackrel{G(k)}{\not\approx} h$ for every $\displaystyle h \in \bigcup_{i=1, i \neq j}^{n-1} X_i$.
Define $G(k+1)$ to be the HNN-extension $\langle G(k),t_{k+1}~\|~ t_{k+1}g_{k+1}t_{k+1}^{-1}=x_j \rangle$. By Lemma \ref{lem:HNN-conj}
the subsets $X_i$, $i=1,\dots,n-1$, will remain independent in $G(k+1)$.

Thus we can assume that $g_{k+1}$ is not commensurable with any element from $\bigcup_{i=1}^{n-1} X_i$ in $G(k)$.
According to the induction hypotheses one can apply Lemma \ref{lem:HNN-conj} to the HNN-extension
$$G(k+1)=\langle G(k),t_{k+1}~\|~ t_{k+1}g_{k+1}t_{k+1}^{-1}=x_1 \rangle$$ to see
that the subsets $X_i \subset G \le G(k+1)$, $i=1,\dots,n-1$, are independent in $G(k+1)$.

Now, set $G_1= \bigcup_{k=0}^\infty G(k)$. Evidently $G_1$ has the required properties.
In the same manner, one can embed $G_1$ into a countable torsion-free group $G_2$ so that each non-trivial element of $G_1$
will be conjugated to $x_i$ in $G_2$, for some $i \in \{1,\dots,n-1\}$, and the subsets $X_i, i=1,\dots,n-1$,
continue to be independent in $G_2$.

Proceeding like that we obtain the desired group $M=\bigcup_{s=1}^\infty G_s$. By the construction, $M$ is a torsion-free countable group
which has exactly $n$ conjugacy classes: $[1], [x_1],\dots, [x_{n-1}]$. The subsets $X_i, i=1,\dots,n-1$, are independent in $M$ because
they are independent in $G_s$ for each $s \in \N$.
\end{proof}

\begin{cor} \label{cor:2-b-s} In Lemma \ref{lem:HNN-emb} one can add that the group $M$ is $2$-boundedly simple.
\end{cor}

\begin{proof} Let a torsion-free countable group $G$ and its non-empty independent subsets $X_i$, $i=1,\dots,n-1$, be as in
Lemma \ref{lem:HNN-emb}. Let $F=F(a_1,\dots,a_{n-1},b_1,\dots,b_{n-1})$ be the free group with the free generating set
$\{a_1,\dots,a_{n-1},b_1,\dots,b_{n-1}\}$, and consider the group $\bar G=G * F$. For each $i=1,\dots,n-1$, define
$$\bar X_i = X_i \cup \{a_i,a_i^{-1}\} \cup \{[a_j,b_i]~|~j=1,\dots,n-1, j\neq i\} \subset \bar G,$$
where $[a_j,b_i]=a_jb_ia_j^{-1}b_i^{-1}$. Using the universal properties of free groups and free products one can easily see
that the subsets $\bar X_i$, $i=1,\dots,n-1$, are independent in $\bar G$.

Now we apply Lemma \ref{lem:HNN-emb} to find a countable torsion-free {\ncc}-group $M$, containing $\bar G$, such that $\bar X_i$,
$i=1,\dots,n-1$, are independent in $M$. Observe that this implies that for any given $i=1,\dots, n-1$,
any two elements of $\bar X_i$ are conjugate in $M$. For arbitrary $x,y \in M\setminus \{1\}$ there exist $i,j \in \{1,\dots,n-1\}$
such that $x \stackrel{M}{\sim} a_i$ and $y \stackrel{M}{\sim} a_j$. If $i=j$ then $x \stackrel{M}{\sim} y$. Otherwise,
$y \stackrel{M}{\sim} a_j \stackrel{M}{\sim} a_j^{-1}$ and $x\stackrel{M}{\sim} [a_j,b_i]$ which is a product of two conjugates of
$a_j$, and, hence, of $y$. Therefore the group $M$ is $2$-boundedly simple, and since $G \le \bar G \le M$, the corollary is proved.
\end{proof}

Below is a particular (torsion-free) case of a theorem proved by Osin in \cite[Thm. 2.6]{Osin-SCT}:

\begin{lemma}\label{lem:fg-emb} Any countable torsion-free group $S$ can be embedded into a $2$-generated group
$M$ so that $S$ is malnormal in $M$ and every element of $M$ is conjugated to an element of $S$ in $M$.
\end{lemma}

\begin{proof} Following Osin's proof of Theorem 2.6 from \cite{Osin-SCT}, we see that the required group $M$ can be
constructed as an inductive
limit of relatively hyperbolic groups $G(i)$, $i \in \N$. More precisely, one sets $G(0)=S*F_2$, where $F_2$ is a free group of rank $2$,
$\xi_0=id_{G(0)}:G(0) \to G(0)$, and
for each $i \in \N$ one constructs a group $G(i)$ and an epimorphism $\xi_i:G(0) \to G(i)$ so that $\xi_i$ is
injective on $S$, $G(i)$ is  torsion-free and hyperbolic relative to
$\{\xi_i(S)\}$, and $\xi_i$ factors through $\xi_{i-1}$.  The group $M$ is defined to be the  direct limit
of $(G(i),\xi_i)$ as $i \to \infty$, i.e., $Q=G(0)/N$ where $N =\bigcup_{i \in \N} \ker(\xi_i)$.
By Lemma \ref{lem:malnorm}, $\xi_i(S)$ is malnormal in $G(i)$, hence the image of $S$ will also be malnormal in $M$.
\end{proof}

\begin{thm} \label{thm:cc-emb} Let $G$ be a torsion-free countable group, $n\in \N$, $n \ge 2$, and non-empty subsets
$X_i \subset G\setminus \{1\}$,  $i =1,\dots,n-1$, be independent in $G$. Then $G$ can be embedded into a 2-generated
torsion-free group $M$ which has
{\ncc}, so that the subsets $X_i, i=1,\dots, n-1$, stay independent in $M$. Moreover, one can choose $M$ to be $2$-boundedly simple.
\end{thm}

\begin{proof} First, according to Corollary \ref{cor:2-b-s},
we can embed the group $G$ into a countable torsion-free group $S$ such that $S$ has {\ncc} and is $2$-boundedly simple,
and $X_i$, $i=1,\dots,n-1$, are independent in $S$. Second, we apply Lemma \ref{lem:fg-emb} to find the $2$-generated group
$M$ from its claim.
Choose any $i,j \in \{1,\dots,n-1\}$, $i \neq j$, and $x \in X_i$, $y \in X_j$. If $x$ and $y$ were commensurable in $M$, the malnormality
of $S$ would imply that $x$ and $y$ must be commensurable in $S$, contradicting the construction. Hence $X_i$, $i=1,\dots,n-1$,
are independent in $M$. Since each element of $M$ is conjugated to an element of $S$, it is evident that $M$ has {\ncc}, is torsion-free
and $2$-boundedly simple.
\end{proof}

\begin{remark} A more direct proof of Theorem \ref{thm:cc-emb}, not using Lemma \ref{lem:fg-emb}, can be extracted from the proof of Theorem
\ref{thm:ext-main} (see Section \ref{sec:add_fg}), applied to the case when $H=M$.
\end{remark}


It is easy to see that Theorem \ref{thm:cc-emb} immediately implies Corollary \ref{cor:non-comm_sep_emb} that was formulated in the
Introduction. As promised, we now give a counterexample to Question \ref{q:2} (formulated in the Introduction) for any $n\ge 3$.

\begin{example} Let $G_2=\langle a,t~\|~tat^{-1}=a^2 \rangle$ be the Baumslag-Solitar $BS(1,2)$-group. Then $G_2$ is torsion-free, and
the elements $t^2,t^4,\dots, t^{2^{n-1}}$ are pairwise non-conjugate in $G_2$
(since this holds in the quotient of $G_2$ by the normal closure of $a$). Suppose that $G_2$ is embedded into a group
$M$ having {\ncc} so that $t^2,t^4,\dots,t^{2^{n-1}}$ are pairwise non-conjugate in $M$. Then $t^2,\dots, t^{2^{n-1}}$
is the list of representatives of
all non-trivial conjugacy classes of $M$. Therefore there exist $k,l \in \{1,\dots,n-1\}$ such that $t \stackrel{M}{\sim} t^{2^k}$ and
$a \stackrel{M}{\sim} t^{2^l}$. Consequently
$$t^2 \stackrel{M}{\sim} t^{2^{k+1}} ~\mbox{ and }~ t^{2^l} \stackrel{M}{\sim} a \stackrel{M}{\sim} a^2 \stackrel{M}{\sim} t^{2^{l+1}},$$
hence $k=l=n-1$ according to the assumptions. But this yields
$$t \stackrel{M}{\sim} t^{2^{n-1}} \stackrel{M}{\sim} a \stackrel{M}{\sim} a^2 \stackrel{M}{\sim} t^2,$$
implying that $t^2 \stackrel{M}{\sim} t^4$, which contradicts our assumptions.

Thus $G_2$ can not be embedded into a {\ncc}-group $M$ in such a way that $t^2,\dots, t^{2^{n-1}}$ remain pairwise non-conjugate in $M$.
\end{example}


\section{Normal subgroups with {\ncc}}\label{sec:norm_sbgps}
If $M$ is a normal subgroup of a group $H$, then $H$ naturally acts on $M$ by conjugation. We shall say that this action
preserves the conjugacy classes of $M$ if for any $h \in H$ and $a \in M$ there exists $b \in M$ such that $hah^{-1}=bab^{-1}$.

\begin{lemma} \label{lem:indep_part} Let $G$ be a torsion-free group,
$N \lhd G$ and $x_1,\dots,x_l \in N \setminus \{1\}$ be pairwise non-commensurable (in $G$)
elements. Then there exists a partition $N\setminus \{1\}=\bigsqcup_{k=1}^l X_k$ of $N\setminus \{1\}$
into a (disjoint) union of $G$-independent subsets $X_1,\dots,X_l$ such that $x_k \in X_k$ for every $k\in \{1,\dots,l\}$. Moreover,
each subset $X_k$ will be invariant under conjugation by elements of $G$.
\end{lemma}

\begin{proof} Since $\stackrel{G}{\approx}$ is an equivalence relation on $G\setminus \{1\}$, one can find the corresponding decomposition:
$G\setminus \{1\}=\bigsqcup_{j \in J} Y_j$, where $Y_j$ is an equivalence class for each $j \in J$. For each $k=1,\dots,l$, there
exists $j(k) \in J$ such that $x_k \in Y_{j(k)}$. Note that $j(k) \neq j(m)$ if $k \neq m$ since $x_k \stackrel{G}{\not\approx} x_m$.

Denote $J'=J\setminus \{j(1),\dots,j(l-1)\}$,
$$X_1=Y_{j(1)} \cap N, \dots, X_{l-1}=Y_{j(l-1)} \cap N,~\mbox{ and }~X_l=\bigcup_{j \in J'} Y_j \cap N.$$
Evidently $N\setminus \{1\}=\bigsqcup_{k=1}^l X_k$, $X_1,\dots, X_l$ are independent subsets of $G$ and $x_k \in X_k$ for each $k=1,\dots,l$.
The final property follows from the construction since for any $a\in G$ and $j \in J$ we have $aY_ja^{-1}=Y_j$ and $aNa^{-1}=N$.
\end{proof}

\begin{lemma} \label{lem:ncc-normal} For every countable group $C$ and each $n \in \N$, $n \ge 2$,
there exists a countable torsion-free group $H$ having a normal
subgroup $M \lhd H$ such that
\begin{itemize}
\item[(i)] $M$ satisfies {\ncc};
\item[(ii)] $M$ is $2$-boundedly simple;
\item[(iii)] the natural action of $H$ on $M$ preserves the conjugacy classes of $M$;
\item[(iv)] $H/M \cong C$.
\end{itemize}
\end{lemma}

\begin{proof} Let $H_0'$ be the free group of infinite countable rank. Choose $N_0' \lhd H_0'$ so that $H_0'/N_0' \cong C$.
Let $F=F(x_1,\dots,x_{n-1})$ denote the free group freely generated by $x_1,\dots, x_{n-1}$. Define
$H_0=H_0' * F$ and let $N_0$ be the normal closure of $N_0' \cup F$ in $H_0$. Evidently,
$H_0/N_0 \cong H_0'/N_0' \cong C$ and
the elements $x_{1},\dots,x_{n-1} \in N_0 \setminus \{1\}$ are pairwise non-commensurable in $H_0$.


By Lemma \ref{lem:indep_part}, one can choose a partition of $N_0\setminus \{1\}$ into the union of $H_0$-independent subsets:
$$N_0 \setminus \{1\}=\bigsqcup_{k=1}^{n-1} X_{0k},$$ so that
$x_{k} \in X_{0k}$ for each $k=1,\dots,n-1$.

By Corollary \ref{cor:2-b-s} there exists a countable torsion-free $2$-boundedly simple group $M_1$ with the
property {\ncc} containing a copy of $N_0$, such that the subsets $X_{0k}$, $k=1,2,\dots,n-1$, are independent in $M_1$.
Denote by $H_1 = H_0*_{N_0} M_1$ the amalgamated product of $H_0$ and $M_1$
along $N_0$, and let $N_1$ be the normal closure of $M_1$ in $H_1$. Note that $H_1$ is torsion-free as an amalgamated product of two
torsion-free groups (\cite[IV.2.7]{L-S}).

We need to verify that the elements $x_{1},\dots,x_{n-1}$ are pairwise non-commen\-surable in $H_1$. Indeed, if $a \in X_{0k}$ and
$b \in X_{0l}$, $k \neq l$,
are conjugate in $H_1$ then there must exist $y_1,\dots,y_t \in M_1 \setminus N_0$ and $z_1,\dots,z_{t-1} \in H_0 \setminus N_0$,
$z_0,z_t \in H_0$ such that
$$z_0y_1\cdots z_{t-1}y_t z_t a z_t^{-1} y_t^{-1} z_{t-1}^{-1} \cdots y_1^{-1}z_0^{-1}\stackrel{H_1}{=}b.$$
Suppose that $t$ is minimal possible with this property.
As conjugation by elements of $H_0$ preserves $X_{0k}$ and $X_{0l}$, we can assume that $z_0,z_t =1$. Hence
$$y_1z_1\cdots z_{t-1}y_t a y_t^{-1} z_{t-1}^{-1} \cdots z_1^{-1}y_1^{-1}b^{-1}\stackrel{H_1}{=}1.$$
By the properties of amalgamated products (see \cite[Ch. IV]{L-S}), the left-hand side in this equality can not be reduced,
consequently $y_t a y_t^{-1} \in N_0 \setminus \{1\}=\bigsqcup_{k=1}^{n-1} X_{0k}$. But then
$y_t a y_t^{-1} \in X_{0k}$ by the properties of $M_1$, contradicting the minimality of $t$. Thus, we have shown that
$x_{k} \stackrel{H_1}{\not\approx} x_{l}$ whenever $k \neq l$.

Assume that the group $H_i=H_{i-1}*_{N_{i-1}} M_i$, $i \ge 1$, has already been constructed, so that
\begin{itemize}
\item[0)] $H_i$ is countable and torsion-free;
\item[1)] $N_{i-1} \lhd H_{i-1}$;
\item[2)] $H_{i-1}=H_0 \cdot N_{i-1}$ and $H_0 \cap N_{i-1}=N_0$;
\item[3)] $M_i$ satisfies {\ncc};
\item[4)] $x_1,\dots,x_{n-1}$ are pairwise non-commensurable in $H_i$.
\end{itemize}

Let $N_i$ be the normal closure of $M_i$ in $H_i$. Because of the condition 4) and Lemma \ref{lem:indep_part},
one can find a partition of $N_i \setminus \{1\}$ into a union of $H_i$-independent subsets:
$$N_i \setminus \{1\}=\bigsqcup_{k=1}^{n-1} X_{ik},$$ so that
$x_{k} \in X_{ik}$ for each $k=1,\dots,n-1$. By Lemma \ref{lem:HNN-emb}
there is a countable group a $M_{i+1}$, with {\ncc}, containing a copy of $N_i$, in which the subsets
$X_{ik}$, $i=1,\dots,n-1$, remain independent. Set $H_{i+1}=H_i *_{N_i} M_{i+1}$. Now, it is easy to verify that
the analogs of the conditions 0)-3) hold for $H_{i+1}$ and
\begin{equation}\label{eq:N_i} N_{i-1}\le M_i \le N_i \le M_{i+1}.
\end{equation}
The analog of the condition 4) is true in $H_{i+1}$ by the same considerations as before (in the case of $H_1$).

Define the group $H=\bigcup_{i=1}^\infty H_i$ and its subgroup $M=\bigcup_{i=1}^\infty N_i$. Observe that the condition $0)$ implies that
$H$ is torsion-free, condition 1)
implies that $M$ is normal in $H$, and 2) implies that $H=H_0 \cdot M$ and $H_0 \cap M=N_0$. Hence $H/M\cong H_0/(H_0 \cap M) \cong C$.
Applying \eqref{eq:N_i} we get $M=\bigcup_{i=1}^\infty M_i$, and thus, by the conditions 3), 4) it enjoys the property
{\ncc}: each element of $M$ will be a conjugate of $x_k$ for some $k \in \{1,\dots,n-1\}$. Since $x_1,\dots,x_{n-1} \in M_1\le M$ and
$M_1$ is 2-boundedly simple, then so will be $M$.
 Finally, 4) implies that
$x_k \stackrel{H}{\nsim} x_l$ whenever $k \neq l$, and, consequently, the natural action of $H$ on $M$ preserves its conjugacy classes. Q.e.d.
\end{proof}

\begin{lemma} \label{lem:HNN-univer} Suppose that $G$ is a group, $N \lhd G$, $A,B \le G$ and $\varphi:A \to B$ is an isomorphism
such that  $\varphi(a) \in aN$ (i.e., the canonical images of $a$ and $\varphi(a)$ in $G/N$ coincide) for each $a \in A$.
Let $L=\langle G, t ~\|~ tat^{-1}=\varphi(a),~\forall~a \in A \rangle $ be the
HNN-extension of $G$ with associated subgroups $A$ and $B$, and let $K$ be the normal closure of $\langle N,t \rangle$ in $L$ .
Then $G \cap K=N$.
\end{lemma}

\begin{proof} This statement easily follows from the universal property of HNN-extensions and is left as an exercise for the reader.
\end{proof}

The next lemma will allow us to construct {\ncc}-groups that are not simple:

\begin{lemma} \label{lem:emb-non-simple} Assume that $H$ is a torsion-free countable group and $M \lhd H$ is a non-trivial normal subgroup.
Then $H$ can be isomorphically embedded into a countable torsion-free group $G$ possessing a normal subgroup $K \lhd G$ such that
\begin{itemize}
\item[1)] $G=HK$ and $H \cap K=M$;
\item[2)] $\forall~x,y \in G\setminus\{1\}$, $\varphi(x)=\varphi(y)$ if and only if $\exists~h \in K$ such that $x=hyh^{-1}$,
where $\varphi:G \to G/K$ is the natural homomorphism; in particular, $K$ will have {\cc};
\item[3)] $\forall~x,y \in G\setminus\{1\}$, $x \stackrel{G}{\sim} y$ if and only if $\varphi(x) \stackrel{G/K}{\sim}\varphi(y)$;
\end{itemize}

\begin{proof}
Choose a set of representatives $Z \subset H$ of cosets of $H$ modulo $M$, in such a way that each coset is represented by a unique element
from $Z$ and $1 \notin Z$.

Define $G(0)$=$H$ and $K(0)=M$. Enumerate the elements of $G(0)\setminus \{1\}$:
$g_1,g_2,\dots$. First we embed the group $G(0)$ into a countable torsion-free group $G_1$,
having a normal subgroup $K_1 \lhd G_1$, such that $G_1=HK_1$, $H \cap K_1=M$ and for every $i \ge 0$ there are $t_i \in K_1$ and
$z_i \in Z$ satisfying $t_{i}g_it_i^{-1} = z_i$.

Suppose that the (countable torsion-free) group $G(j)$, $j \ge 0$, and $K(j) \lhd G(j)$, have already been constructed so that $H \le G(j)$,
$G(j)=HK(j)$, $H \cap K(j)=M$ and, if $j \ge 1$, then $t_{j}g_jt_j^{-1} = z_j$ for some $t_j \in K(j)$ and $z_j \in Z$.
The group $G(j+1)$, containing $G(j)$,
is defined as the following HNN-extension:
$$G(j+1)=\langle G(j),t_{j+1}~\|~t_{j+1}g_{j+1}t_{j+1}^{-1}=z_{j+1}\rangle,$$ where
$z_{j+1} \in Z \subset H$ is the unique representative satisfying $g_{j+1} \in z_{j+1}K(j)$ in $G(j)$. Denote by $K(j+1)\lhd G(j+1)$  the normal
closure of $\langle K(j),t_{j+1} \rangle$ in $G(j+1)$. Evidently the group $G(j+1)$ is countable and torsion-free,
$H \le G(j) \le G(j+1)$, $G(j+1)=HK(j+1)$ and $H \cap K(j+1)=H\cap K(j)=M$ by Lemma \ref{lem:HNN-univer}.

Now, it is easy to verify that the group $G_1=\bigcup_{j=0}^\infty G(j)$ and its normal subgroup $K_1=\bigcup_{j=0}^\infty K(j)$ enjoy
the required properties.

In the same way we can embed $G_1$ into a countable torsion-free group $G_2$,
that has a normal subgroup $K_2 \lhd G_2$, so that $G_2=HK_2$, $H \cap K_2=M$ and
each element of $G_1\setminus \{1\}$ is conjugated in $G_2$ to a corresponding element of $Z$. Performing such a procedure infinitely many times
we achieve the group $G=\bigcup_{i=1}^\infty G_i$ and a normal subgroup $K=\bigcup_{i=1}^\infty K_i \lhd G$ that satisfy the
claims 1) and 2) of the lemma. It is easy to see that the claim 2) implies 3), thus the proof is finished.
\end{proof}

\end{lemma}

\section{Adding finite generation}\label{sec:add_fg}

\begin{thm} \label{thm:ext-main} Assume that $H$ is a countable torsion-free group and $M$ is a non-trivial normal subgroup of $H$.
Let $F$ be an arbitrary non-elementary torsion-free word hyperbolic group. Then there exist a countable torsion-free group $Q$,
containing $H$, and a normal subgroup $N \lhd Q$ with the following properties:
\begin{itemize}
\item[1.] $H$ is malnormal in $Q$;
\item[2.] $Q=H \cdot N$ and $N \cap H =M$;
\item[3.] $N$ is a quotient of $F$;
\item[4.] the centralizer $C_Q(N)$ of $N$ in $Q$ is trivial;
\item[5.] for every $q \in Q$ there is $z \in H$ such that $q \stackrel{Q}{\sim} z$.
\end{itemize}
\end{thm}

\begin{proof}
The group $Q$ will be constructed as a direct limit of relatively hyperbolic groups.

Step \fbox{0}. Set $G(0)=H * F$ and $F(0)=F$; then $G(0)$ is hyperbolic relative to
its subgroup $H$  and $F(0)$ is a suitable subgroup of $G(0)$ by Lemma \ref{lem:hyp_suit_sbgp_free_prod}.
Let $N(0) \lhd G(0)$ be the normal closure of the subgroup $\langle M, F \rangle$
in $G(0)$.
Evidently $G(0)=H\cdot N(0)$ and $H \cap N(0)=M$. Enumerate all the elements of $N(0)$: $\{g_0,g_1,g_2,\dots\}$, and
of $G(0)$: $\{q_0,q_1,q_2,\dots\}$, in such a way that $g_0=q_0=1$.

Steps \fbox{0-i}. Assume the groups $G(j)$, $j=0,\dots,i$, $i \ge 0$,
have been already constructed, so that
\begin{itemize}
\item[$1^\circ$.] for each $1\le j \le i$ there is an epimorphism $\psi_{j-1}: G(j-1) \to G(j)$ which is injective on (the image of)
$H$ in $G(j-1)$. Denote $F(j)=\psi_{j-1}(F(j-1))$, $N(j)=\psi_{j-1}(N(j-1))$;
\item[$2^\circ$.] $G(j)$ is torsion-free and hyperbolic relative to (the image of) $H$, and $F(j) \le G(j)$
is a suitable subgroup, $j=0,\dots, i$;
\item[$3^\circ$.] $G(j)=H\cdot N(j)$, $N(j)\lhd G(j)$ and $H \cap N(j)=M$, $j=0,\dots, i$;
\item[$4^\circ$.] the natural image $\bar g_j$ of $g_j$  in $G(j)$ belongs to $F(j)$, $j=0,\dots, i$;
\item[$5^\circ$.] there exists $z_j \in H$ such that $\bar q_j \stackrel{G(j)}{\sim} z_j$, $j=0,\dots, i$, where $\bar q_j$
is the image of $q_j$ in $G(j)$.
\end{itemize}

Step \fbox{i+1}. Let $\hat q_{i+1} \in G(i)$, $\hat g_{i+1} \in N(i)$ be the images of $q_{i+1}$ and $g_{i+1}$ in $G(i)$.
First we construct the group $G(i+1/2)$, its normal subgroup $K_{i+1}$ and its element $t_{i+1}$ as follows.

If for some $f \in G(i)$, $f\hat q_{i+1} f^{-1}=z \in H$, then
set $G(i+1/2)=G(i)$, $K_{i+1}=N(i)\lhd G(i+1/2)$ and
$t_{i+1}=1$.

Otherwise, $\hat q_{i+1}$ is a hyperbolic element of infinite order in $G(i)$. Since $G(i)$ is torsion-free, the elementary
subgroup $E_{G(i)}(\hat q_{i+1})$ is cyclic, thus $E_{G(i)}(\hat q_{i+1}) = \langle h x \rangle$ for some $h \in H$ and $x \in N(i)$
(by $3^\circ$), and $\hat q_{i+1}=(hx)^m$ for some $m \in \Z$.
Now, by Lemma \ref{lem:Eg}, $G(i)$ is hyperbolic relative to $\{H,\langle hx \rangle\}$. Choose $y \in M$ so that
$hy \neq 1$ and let $G(i+1/2)$ be the following HNN-extension of $G(i)$:
$$G(i+1/2) = \langle G(i),t_{i+1}~\|~t_{i+1} (hx) t_{i+1}^{-1} = hy\rangle. $$
The group $G(i+1/2)$ is torsion-free and hyperbolic relative to $H$ by Lemma \ref{lem:HNN-rel_hyp}.

Let us now verify that the subgroup $F(i)$ is
suitable in $G(i+1/2)$. Indeed, according to Lemma \ref{lem:suit-inf_many_non-comm}, there are two hyperbolic
elements $f_1,f_2 \in F(i)$ of infinite order in $G(i)$ such that $f_l \stackrel{G(i)}{\not\approx} hx$,
$f_l \stackrel{G(i)}{\not\approx} hy$, $l=1,2$, and
$f_1 \stackrel{G(i)}{\not\approx} f_2$. Then $f_1 \stackrel{G(i+1/2)}{\not\approx} f_2$ by Lemma \ref{lem:HNN-conj}. It remains to check
that $f_l$ is a hyperbolic element of $G(i+1/2)$ for each $l=1,2$. Choose an arbitrary element $w \in H$ and observe that
$f_l \stackrel{G(i)}{\not\approx} w$ (since $H$ is malnormal in $G(i)$ by Lemma \ref{lem:malnorm}, a non-trivial power of $f_l$ is conjugated to
an element of $H$ if and only if $f_l$ is conjugated to an element of $H$ in $G(i)$, but the latter is impossible because
$f_l$ is hyperbolic in $G(i)$). Applying Lemma \ref{lem:HNN-conj} again, we get that $f_l \stackrel{G(i+1/2)}{\not\approx} w$ for any
$w \in H$. Hence $f_1,f_2 \in F(i)$ are hyperbolic elements of infinite order in $G(i+1/2)$. The intersection
$E_{G(i+1/2)}(f_1) \cap E_{G(i+1/2)}(f_2)$ must be finite, since these groups are virtually cyclic (by Lemma \ref{lem:Eg}), and
$f_1$ is not commensurable with $f_2$ in $G(i+1/2)$. But $G(i+1/2)$ is torsion-free, therefore
$E_{G(i+1/2)}(f_1) \cap E_{G(i+1/2)}(f_2)=\{1\}$. Thus $F(i)$ is a suitable subgroup of $G(i+1/2)$.

Lemma \ref{lem:HNN-univer} assures that $H \cap K_{i+1}=M$ where $K_{i+1} \lhd G(i+1/2)$ is the normal closure of
$\langle N(i),t_{i+1} \rangle$ in $G(i+1/2)$. Finally,  note that
$$t_{i+1} \hat q_{i+1} t_{i+1}^{-1}=t_{i+1} (hx)^m t_{i+1}^{-1}=(hy)^m =z \in H~\mbox{ in } G(i+1/2).$$

Now, that the group $G(i+1/2)$ has been constructed, set $T_{i+1}=\{\hat g_{i+1}, t_{i+1}\}\subset K_{i+1}$ and define $G(i+1)$ as follows.
Since $T_{i+1} \cdot F(i) \subset K_{i+1} \lhd G(i+1/2)$,
we can apply Theorem \ref{thm:main_SCT} to find a group $G(i+1)$ and an epimorphism $\varphi_i:G(i+1/2) \to G(i+1)$ such that
$\varphi_i$ is injective on $H$, $G(i+1)$ is torsion-free and hyperbolic relative to (the image of) $H$,
$\{\varphi_i(\hat g_{i+1}), \varphi_i(t_{i+1})\} \subset \varphi_i(F(i))$, $\varphi_i(F(i))$ is a suitable subgroup of $G(i+1)$,
and $\ker(\varphi_i) \le K_{i+1}$. Denote by $\psi_i$ the
restriction of $\varphi_i$ on $G(i)$. Then $\psi_i(G(i))=\varphi_i(G(i))=G(i+1)$ because $G(i+1/2)$ was generated by $G(i)$ and
$t_{i+1}$, and according to the construction,
$t_{i+1} \in \varphi_i(F(i))\le \varphi_i(G(i))$. Now, after defining $F(i+1)=\psi_{i}(F(i))$, $N(i+1)=\psi_{i}(N(i))$,
$\bar g_{i+1}=\varphi_i(\hat g_{i+1}) \in F(i+1)$ and $z_{i+1}=\varphi_i(z) \in H$, we see that the conditions
$1^\circ$,$2^\circ$,$4^\circ$ and $5^\circ$ hold in the case when $j=i+1$. The properties $G(i+1)=H\cdot N(i+1)$ and
$N(i+1)\lhd G(i+1)$ are immediate consequences of their analogs for $G(i)$ and $N(i)$. Finally, observe that
\begin{multline*}\varphi_i^{-1}(H \cap N(i+1))=H\cdot \ker(\varphi_i) \cap N(i) \cdot \ker(\varphi_i) =
\bigl(H \cap N(i) \cdot \ker(\varphi_i)\bigr) \cdot \ker(\varphi_i) \\
\subseteq  \bigl(H \cap K_{i+1} \bigr) \cdot \ker(\varphi_i)= M \cdot \ker(\varphi_i).
\end{multline*}
Therefore $H \cap N(i+1)=M$ and the condition $3^\circ$ holds for $G(i+1)$.

Let $Q=G(\infty)$ be the direct limit of the sequence $(G(i),\psi_i)$ as $i \to \infty$, and let $F(\infty)$ and $N=N(\infty)$ be the limits
of the corresponding subgroups. Then $Q$ is torsion-free by $2^\circ$,
$N\lhd Q$, $Q=H \cdot N$ and $H \cap N=M$ by $3^\circ$. $N \le F(\infty)$ by $4^\circ$, and $5^\circ$ implies
the condition $5$ from the claim.


Since $F(0) \le N(0)$ we get $F(\infty) \le N$.
Thus $N=F(\infty)$ is a homomorphic image of $F(0)=F$.

For any $i,j \in \N \cup \{\infty\}$, $i<j$, we have a natural epimorphism $\zeta_{ij}: G(i) \to G(j)$ such that
if $i<j<k$ then $\zeta_{jk} \circ \zeta_{ij}=\zeta_{ik}$. Take any $g \in G(0)$.
Since $F=F(0)$ is finitely generated, using the properties of direct limits one can show that if
$w=\zeta_{0 \infty}(g) \in C_Q(F(\infty))$ in $Q$, then $\zeta_{0j}(g) \in C_{G(j)}(F(j))$ for some $j\in\N$.
But $C_{G(j)}(F(j)) \le E_{G(j)}(F(j))=\{1\}$ (by formulas \eqref{eq:elem} and \eqref{eq:E_G})
because $F(j)$ is a suitable subgroup of $G(j)$,
hence $w= \zeta_{j\infty}\bigl(\zeta_{0j}(g)\bigr)=1$, that is,
$C_Q(F(\infty))=C_Q(N)=\{1\}$.
This concludes the proof.
\end{proof}

The next statement is well-known: 
\begin{lemma} \label{lem:out-gp} Assume $G$ is a group and $N \lhd G$ is a normal subgroup such that $C_G(N)\subseteq N$, where
$C_G(N)$ is the centralizer  of $N$ in $G$. Then the quotient-group $G/N$ embeds into the outer automorphism group
$Out(N)$.
\end{lemma}

\begin{proof} The action of $G$ on $N$ by conjugation induces a natural homomorphism $\varphi$
from $G$ to the automorphism group $Aut(N)$ of $N$. Since $\varphi(N)$ is exactly the group of inner automorphisms $Inn(N)$ of $N$,
one can define a new homomorphism $\bar \varphi: G/N \to Out(N)=Aut(N)/Inn(N)$
in the natural way: $\bar \varphi(gN)= \varphi(g)Inn(N)$ for every $gN \in G/N$.
It remains to check that $\bar \varphi$ is injective, i.e., if $g \in G \setminus N$ then $\bar\varphi(gN) \neq 1$ in $Out(N)$;
or, equivalently, $\varphi(g) \notin Inn(N)$. Indeed, otherwise there would exist $a \in N$ such that $ghg^{-1}=aha^{-1}$
for every $h \in N$, thus $N \not\ni a^{-1}g \in C_G(N)$, contradicting the assumptions. Q.e.d.
\end{proof}

Note that for an arbitrary group $N$, any subgroup $C \le Out(N)$ naturally
acts on the set of conjugacy classes $\mathfrak{C}(N)$ of the group $N$.

\begin{thm} \label{thm:out-emb} For any $n \in \N$, $n \ge 2$, and an arbitrary countable group $C$, $C$ can be isomorphically embedded
into the outer automorphism group $Out(N)$ of a group $N$ satisfying the following conditions:
\begin{itemize}
\item $N$ is torsion-free;
\item $N$ is generated by two elements;
\item $N$ has {\ncc} and the natural action  of $C$ on $\mathfrak{C}(N)$ is trivial;
\item $N$ is $2$-boundedly simple.
\end{itemize}
\end{thm}

\begin{proof} By Lemma \ref{lem:ncc-normal} we can find a countable torsion-free group $H$ and its normal subgroup $M$ enjoying the
properties (i)-(iv) from its claim. Now, if $F$ denotes the free group of rank $2$, we can obtain a countable torsion-free group $Q$
together with its normal subgroup $N$ that satisfy the conditions $1$-$5$ from the statement of Theorem \ref{thm:ext-main}.

Then $N$ is torsion-free and generated by two elements (as a quotient of $F$). Condition $2$ implies that $Q/N \cong H/M \cong C$ and,
by $4$ and Lemma \ref{lem:out-gp}, $C$ embeds into the group $Out(N)$.

Using property $5$, for each $g \in N$ we can find $u \in Q$ and $z \in H$ such that $ugu^{-1}=z \in N \cap H=M$. Since $Q=HN$, there
are $h \in H$ and $x \in N$ such that $u=hx$. Since $z,h^{-1}zh \in M$ and the action of $H$ on $M$ preserves the conjugacy classes of $M$,
there is $r \in M$ such that $rh^{-1}zhr^{-1}=z$, hence $z=rh^{-1}ugu^{-1}(rh^{-1})^{-1}=rxgx^{-1}r^{-1}$, where $v=rx \in N$.
Thus for every $g \in N$ there is $v \in N$ such that $vgv^{-1} \in M$. Evidently, this implies that $N$ is also $2$-boundedly simple.
Since $M$ has {\ncc}, the number of conjugacy classes in $N$ will
be at most $n$.

Suppose $x_1,x_2 \in M$ and $x_1 \stackrel{M}{\nsim} x_2$. Then $x_1 \stackrel{H}{\nsim} x_2$ (by the property (iii)
from the claim of Lemma \ref{lem:ncc-normal}), and since $H$ is malnormal in $Q$ we get $x_1 \stackrel{Q}{\nsim} x_2$.
Hence $x_1 \stackrel{N}{\nsim} x_2$, i.e., $N$ also enjoys {\ncc}.

The fact that the natural action of $C$ on $\mathfrak{C}(N)$ is trivial follows from the same property for the action of $H$ on
$\mathfrak{C}(M)$ and the malnormality of $H$ in $Q$. Q.e.d.
\end{proof}

Now, let us proceed with the
\begin{proof}[Proof of Theorem \ref{thm:emb-ext-non-simple}.] First we apply Lemma \ref{lem:emb-non-simple}
to construct a group $G$ and a normal subgroup $K \lhd G$
according to its claim. Now, by Theorem \ref{thm:ext-main}, there is a group $Q$, having a normal subgroup $N \lhd Q$ such that
$G$ is malnormal in $Q$, $Q=GN$, $G\cap N=K$, $rank(N) \le 2$ (if one takes the free group of rank $2$ as $F$)
and every element $q \in Q$ is conjugated (in $Q$) to an element of $G$.
By claim 2) of Lemma \ref{lem:emb-non-simple}, $K$ has {\cc}, and an argument, similar to the one used in the proof
of Theorem \ref{thm:out-emb}, shows that $N$ will also have {\cc}. Consequently, $rank(N) > 1$ because $N$ is torsion-free,
hence $rank(N)=2$.

Since $G=HK$ and $H \cap K=M$ we have $Q=HKN=HN$ and $H \cap N=H \cap K=M$. Since $Q/N \cong H/M$ and $N$ can be generated by
two elements, we can conclude that $rank(Q) \le rank(H/M)+2$.

Consider arbitrary $x,y \in Q\setminus \{1\}$ and suppose that $\varphi(x) \stackrel{Q/N}{\sim} \varphi(y)$.
By Theorem \ref{thm:ext-main}, there are $w,z \in G\setminus \{1\}$ such that $x \stackrel{Q}{\sim} w$ and $y \stackrel{Q}{\sim} z$.
Therefore $\varphi(w) \stackrel{Q/N}{\sim} \varphi(z)$, hence the images of $w$ and $z$ in $G/K$ are also conjugate.
By claim 3) of Lemma \ref{lem:emb-non-simple}, $w \stackrel{G}{\sim} z$, implying $x \stackrel{Q}{\sim} y$.
\end{proof}

Theorem \ref{thm:emb-ext-non-simple} provides an alternative way of obtaining torsion-free groups that have finitely many conjugacy classes:
for any countable group $C$ we can choose a free group $H$ of countable rank and a normal subgroup $\{1\}\neq M \lhd H$ so that
$H/M \cong C$, and then apply Theorem \ref{thm:emb-ext-non-simple} to the pair $(H,M)$ to get

\begin{cor}\label{cor:ncc-non-simple} Assume that $n \in \N$, $n \ge 2$, and $C$ is a countable group that contains exactly
$(n-1)$ distinct conjugacy classes. Then there exists a torsion-free group $Q$ and $N \lhd Q$ such that
\begin{itemize}
\item $Q/N \cong C$;
\item $N$ has {\cc} and $Q$ has {\ncc};
\item $rank(N)=2$ and $rank(Q) \le rank(C)+2$.
\end{itemize}
\end{cor}

\begin{cor} \label{cor:klein-bottle-emb} The group $G_1$, given by presentation \eqref{eq:Kl-bot}, can be isomorphically embedded into a
$2$-generated torsion-free group $Q$ satisfying {\rm ($4$CC)} in such a way that $t \stackrel{Q}{\nsim} t^{-1}$.
\end{cor}

\begin{proof} Denote by $K$ the kernel of the homomorphism $\varphi: G_1 \to \Z_3$,
for which $\varphi(a)=0$ and $\varphi(t)=1$, where $\Z_3$ is the group of integers modulo $3$. Now, apply Theorem \ref{thm:emb-ext-non-simple}
to the pair $(G_1,K)$ to find the group $Q$, containing $G_1$, and the normal subgroup $N \lhd Q$ from its claim.
Since $Q/N \cong G_1/K \cong\Z_3$ has ($3$CC), the group $Q$ will have ($4$CC).
We also have $t \stackrel{Q}{\nsim} t^{-1}$ because the images of $t$ and $t^{-1}$ are not conjugate in $Q/N$.

Choose an element $q_1\in Q\setminus N$. Then $q_2=q_1^3 \in N \setminus \{1\}$ and since $N$ is $2$-generated and
has {\cc}, there is $q_3 \in N$ such that $N=\langle q_2,q_3 \rangle$ in $Q$. As $Q/N$ is generated by the image of
$q_1$, the group $Q$ will be generated by $\{q_1,q_2,q_3\}$, and, consequently, by $\{q_1,q_3\}$. Q.e.d.
\end{proof}

\section{Combinatorics of paths in relatively hyperbolic groups}\label{sec:comb_paths}
Let $G$ be a group hyperbolic relative
to a family of proper subgroups $\{H_\lambda\}_{\lambda \in \Lambda}$, and let $\mathcal{X}$ be a finite symmetrized
relative generating set of $G$. Denote $\mathcal{H}=\bigcup_{\lambda \in \Lambda} \left( H_\lambda \setminus \{1\} \right)$.

For a  combinatorial path $p$ in the Cayley graph {\ga} (of
$G$ with respect to $\cX\cup \cH$) $p_-$, $p_+$, $\L(p)$, and
$lab (p)$ will denote the initial point, the ending point, the length
(that is, the number of edges) and the label of $p$ respectively. $p^{-1}$
will be the path obtained from $p$ by following it in the reverse direction.
Further, if $\Omega$ is a subset of $G$ and $g \in \langle \Omega \rangle \le G$, then
$|g|_\Omega$ will be used to denote the length of a shortest word in $\Omega^{\pm 1}$ representing
$g$.

We will be using the following terminology from \cite{Osin-RHG}. Suppose $q$ is a
path in {\ga}.
A subpath $p$ of $q$ is called an {\it $H_\lambda $-component}
for some $\lambda \in \Lambda $ (or simply a {\it component}) of
$q$, if the label of $p$ is a word in the alphabet
$H_\lambda\setminus \{ 1\} $ and $p$ is not contained in a bigger
subpath of $q$ with this property.

Two components $p_1, p_2$ of a path $q$ in {\ga} are called {\it
connected} if they are $H_\lambda $-components for the same
$\lambda \in \Lambda $ and there exists a path $c$ in {\ga}
connecting a vertex of $p_1$ to a vertex of $p_2$ such that ${lab
(c)}$ entirely consists of letters from $ H_\lambda $. In
algebraic terms this means that all vertices of $p_1$ and $p_2$
belong to the same coset $gH_\lambda $ for a certain $g\in G$.
 We can always assume $c$ to have length at most $1$, as
every nontrivial element of $H_\lambda $ is included in the set of
generators.  An $H_\lambda $-component $p$ of a path $q$ is
called {\it isolated } if no other $H_\lambda $-component of
$q$ is connected to $p$.

The next statement is a particular case of Lemma 2.27 from \cite{Osin-RHG}; we shall formulate it in a slightly more general
form, as it appears in \cite[Lemma 2.7]{Osin-periph}:



\begin{lemma}\label{lem:omega}
Suppose that a group $G$ is hyperbolic relative to a family of
subgroups $\Hl $. Then there exists a finite subset $\Omega
\subseteq G$ and a constant $K \in \N$ such that the following
holds. Let $q$ be a cycle in {\ga}, $p_1, \ldots , p_k$ be
a collection of isolated components of $q$ and  $g_1, \ldots , g_k$ be the elements of $G$
represented by $\lab(p_1), \ldots, \lab(p_k)$
respectively. Then $g_1, \ldots , g_k$ belong to the subgroup
$\langle \Omega\rangle \le G$  and the word lengths of $g_i$'s
with respect to $\Omega$ satisfy
$$ \sum\limits_{i=1}^k |g_i|_{\Omega} \le K\L(q).$$
\end{lemma}

\begin{df} Suppose that  $m \in \N$  and
$\Omega$ is a finite subset of $G$. Define $\cW (\Omega,m)$ to be the set of all
words $W$ over the alphabet $\cX \cup \cH$ that have the following form:
$$W \equiv x_0h_0x_1h_1 \dots x_l h_l x_{l+1},$$ where $ l \in \Z$, $l \ge -2$ (if $l=-2$ then $W$ is the empty word;
if $l=-1$ then $W \equiv x_0$),
$h_i$ and $x_i$ are considered as single letters and
\begin{itemize}
\item[1)] $x_i \in \cX \cup \{1\}$, $i=0,\dots,l+1$,
and for each $i=0,\dots,l$, there exists $\lambda(i) \in \Lambda$ such that $h_i \in H_{\lambda(i)}$;
\item[2)] if $\lambda(i)=\lambda(i+1)$ then $x_{i+1} \notin H_{\lambda(i)}$ for each $i=0,\dots,l-1$;
\item[3)] $h_i \notin \{h \in \langle \Omega \rangle ~:~|h|_\Omega \le m \}$, $i=0,\dots,l$.
\end{itemize}
\end{df}

Choose the finite subset $\Omega \subset G$ and the constant $K>0$ according to the claim of Lemma \ref{lem:omega}.

Recall that a path $q$ in {\ga} is said to be {\it without
backtracking} if all of its components are isolated.

\begin{lemma} \label{lem:no_back} Let $q$ be a path in the Cayley graph {\ga} with $\lab(q) \in \cW (\Omega,m)$ and $m \ge 5K$.
Then $q$ is without backtracking.
\end{lemma}

\begin{proof} Assume the contrary to the claim. Then one can choose a path $q$ providing a counterexample
of the smallest possible length. Thus if $p_1,\dots, p_l$ is the (consecutive) list of all components of $q$
then $l \ge 2$, $p_1$ and $p_l$ must be connected $H_{\lambda'}$-components, for some $\lambda' \in \Lambda$,
the components $p_2,\dots,p_{l-1}$ must be isolated, and $q$ starts with $p_1$ and ends with $p_l$.
Since $\lab(q) \in \cW(\Omega,m)$ we have $\L(q)\le 2l-1$.

If $l=2$ then the $(\cX \cup \{1\})$-letter between $p_1$ and $p_2$ would belong to $H_{\lambda'}$ contradicting
the property 2) from the definition of $\cW (\Omega,m)$.

Therefore $l\ge 3$.
Since $p_1$ and $p_l$ are connected, there exists a path $v$ in {\ga} between $(p_l)_-$ and $(p_1)_+$ with
$\lab(v) \in H_{\lambda'}$ (thus we can assume that $\L(v)\le 1$).
Denote by $\hat q$ the subpath of $q$ starting with $(p_1)_+$ and ending with $(p_l)_-$.
Note that $\L(\hat q) = \L(q)-2\le 2l-3$, and $p_2, \dots, p_{l-1}$ is the list of components of $\hat q$, all of which are isolated.
If one of them were
connected to $v$ it would imply that it is connected to $p_1$ contradicting with the minimality of $q$.
Hence the cycle $o=\hat q v$ possesses $k= l-2\ge 1$ isolated components,
which represent elements $h_1,\dots,h_k \in \cH$.
Consequently, applying Lemma \ref{lem:omega} one obtains that $h_i \in \langle \Omega \rangle$, $i=1,\dots,k$, and
$$\sum_{i=1}^k |h_i|_\Omega \le K\L(o) \le K(\L(\hat q)+1) \le K(2l-2).$$
By the condition 3) from the definition of $\cW (\Omega,m)$ one has $|h_i|_\Omega>m \ge 5K$ for each $i=1,\dots,k$. Hence
$$k \cdot 5 K \le \sum_{i=1}^k |h_i|_\Omega \le K(2l-2), \mbox{ or } 5 \le \frac{2l-2}{k},$$
which contradicts the inequality $k \ge l-2$. Q.e.d.
\end{proof}

\begin{df}
Consider an arbitrary cycle $o=rqr'q'$ in {\ga}, where $\lab(q)$ and
$\lab(q')$ belong to $\cW (\Omega,m)$. Let
$p$ be a component of $q$ (or $q'$). We will say that $p$ is {\it
regular} if it is not an isolated component of $o$. If $m\ge 5K$, and hence $q$ and
$q'$ are without backtracking by Lemma \ref{lem:no_back}, this means that $p$ is either
connected to some component of $q'$ (respectively $q$), or to a
component of $r$ or $r'$.
\end{df}

\begin{lemma}\label{lem:regul} In the above notations,
suppose that $m\ge 7K$ and denote $C=\max\{ \L(r),\L(r')\}$. Then
\begin{itemize}
\item[\rm (a)] if $C\le 1$ then every component of $q$ or $q'$ is regular;
\item[\rm (b)] if $C\ge 2$ then each of $q$ and $q'$ can have at most $4C$ components which are not regular.
\item[\rm (c)] if $l$ is the number of components of $q$, then at least $(l-6C)$ of
components of $q$ are connected to components of $q'$; and two distinct components of $q$ can not be connected to the same
component of $q'$. Similarly for $q'$.
\end{itemize}
\end{lemma}

\begin{proof} Assume the contrary to (a). Then one can choose a cycle $o=rqr'q'$  with
$\L(r),\L(r') \le 1$, having at least one isolated
component on $q$ or $q'$, and such that
$\L(q)+\L(q')$ is minimal. Clearly the latter condition implies that
each component of $q$ or $q'$ is an isolated component of $o$.
Therefore $q$ and $q'$ together contain $k$ distinct isolated components of
$o$, representing elements $h_1,\dots,h_k \in \cH$,
where $k \ge 1$ and $k\ge (\L(q)-1)/2+
(\L(q')-1)/2$. Applying Lemma \ref{lem:omega} we obtain $h_i\in \langle \Omega \rangle$, $i=1,\dots,k$,
and $$\sum_{i=1}^k |h_i|_\Omega \le K\L(o) \le K(\L(q)+\L(q')+2).$$
Recall that $|h_i|_\Omega >m\ge 7 K$ by the property 3) from the definition of  $\cW (\Omega,m)$. Therefore
$\sum_{i=1}^k |h_i|_\Omega \ge k \cdot 7 K$, implying
$$7\le \frac{2}{k} \left( \frac{\L(q)}{2}+\frac{\L(q')}{2}+1\right)\le
\frac{2}{k} \left( \frac{\L(q)-1}{2}+\frac{\L(q')-1}{2}+2\right)\le 6,$$
which yields a contradiction.

Let us prove (b). Suppose that $C \ge 2$ and $q$ contains more than
$4C$ isolated components of $o$.
We shall consider two cases:

{\bf Case 1}. No component of $q$ is connected to a component of
$q'$. Then a component of $q$ or $q'$ can be regular only if it is
connected to a component of $r$ or $r'$. Since, by Lemma \ref{lem:no_back}, $q$ and $q'$ are
without backtracking, two distinct components of $q$ or $q'$ can
not be connected to the same component of $r$ (or $r'$). Hence $q$
and $q'$ together can contain at most $2C$ regular components.
Thus the cycle $o$
has $k$ isolated components, representing elements $h_1,\dots,h_k \in \cH$,
where $k\ge 4C > 4$ and
$k \ge (\L(q)-1)/2 + (\L(q')-1)/2-2C$.
By Lemma \ref{lem:omega}, $h_i \in \langle \Omega \rangle$ for each $i=1,\dots,k$,
and $\sum_{i=1}^k|h_i|_\Omega \le K(\L(q)+\L(q')+2C)$. Once again we can use the property 3) from the definition of
$\cW (\Omega,m)$ to achieve
\begin{multline*} 7 \le \frac{2}{k}\left( \frac{\L(q)}{2}+\frac{\L(q')}{2}+\frac{2C}{2}\right)\le
 \frac{2}{k} \left(\frac{\L(q)-1}{2}+\frac{\L(q')-1}{2}-2C+1+3C\right)\le \\
\frac{2}{k} \left( \frac{\L(q)-1}{2}+\frac{\L(q')-1}{2}-2C \right)
+\frac2k + \frac{6C}{k} \le 2+ \frac12 + \frac32=4,
\end{multline*} yielding a contradiction.

{\bf Case 2.} The path $q$ has at least one component which is
connected to a component of $q'$. Let $p_1,\dots,p_{l}$ denote
the sequence of all components of $q$. By part (a), if $p_{s}$ and
$p_{t}$, $1 \le s \le t \le l$, are connected to components of
$q'$, then for any $j$, $s \le j \le t$, $p_j$ is connected to some component of $q'$ (because $q$ is without backtracking
by Lemma \ref{lem:no_back}). We can
take $s$ (respectively $t$) to be minimal (respectively maximal)
possible. Consequently $p_1,\dots,p_{s-1}, p_{t+1},\dots,p_{l}$
will contain the set of all isolated components of $o$ that belong
to $q$, and none of these components will be connected to a component of $q'$.

Without loss of generality we may assume that $s-1 \ge 4C/2=2C$.
Since $p_s$ is connected to some component $p'$ of $q'$, there
exists a path $v$ in {\ga} satisfying $v_-=(p_{s})_-$,
$v_+=p'_+$, $\lab(v) \in \mathcal{H}\cup \{1\}$, $\L(v)\le 1$. Let $\bar q$
(respectively $\bar q'$) denote the subpath of $q$ (respectively
$q'$) from $q_-$ to $(p_s)_-$ (respectively from $p'_+$ to
$q'_+$). Consider a new cycle $\bar o = r \bar q v \bar q'$.
Reasoning as before, one can show that
$\bar o$ has $k$ isolated components, 
where $k \ge 2C \ge 4$ and
$k \ge (\L(\bar q)-1)/2 + (\L(\bar q')-1)/2 -C-1$.
Now, an application of Lemma \ref{lem:omega} to the cycle
$\bar o$ together with the property 3) from the definition of $\cW (\Omega,m)$ will lead to a contradiction as before.

By the symmetry, the statement (b) of the lemma also holds for $q'$.

The claim (c) follows from (b) and the estimate $\L(r)+\L(r') \le 2C$ because if two different
components $p$ and $\bar p$ of $q$ were connected to the same component of some path in {\ga}, then $p$ and $\bar p$ would also
be connected with each other, which would contradict Lemma \ref{lem:no_back}.
\end{proof}

\begin{lemma} \label{lem:conseq-reg-strong} In the previous notations, let $m \ge 7K$, $C=\max\{ \L(r),\L(r')\}$,
and let $p_1,\dots,p_l$, $p_1',\dots,p_{l'}'$ be the consecutive lists of the components of $q$ and $q'^{-1}$ respectively
If $l \ge 12\max\{C,1\}+2$, then there are indices $s,t,s' \in \N$ such that $1\le s \le 6C+1$, 
$ l-6\max\{C,1\} \le t \le l$ and  for every 
$i \in \{0,1,\dots,t-s\}$, the component $p_{s+i}$ of $q$ is connected to the component $p'_{s'+i}$ of $q'$.
\end{lemma}

\begin{proof} 
By part (c) of Lemma \ref{lem:regul},
there exists $s \le 6C+1$ such that the component $p_s$ is connected to a component $p'_{s'}$  for
some $s' \in \{1,\dots,l'\}$. Thus there is a path $r_1$ between $(p'_{s'})_+$ and $(p_s)_+$ with $\L(r_1) \le 1$.
Consider a new cycle $o_1=r_1q_1r'q_1'$ where $q_1$ is the segment of $q$ from $(p_s)_+$ to $q_+=r'_-$ and $q_1'$ is the segment
of $q'$ from $q'_-=r'_+$ to $(p'_{s'})_+$.

\begin{figure}[!ht]
  \begin{center}
   \input{figure1.pstex_t}

  \end{center}
  \caption{}\label{fig:1}
\end{figure}

Observe that $p_{s+1},\dots, p_l$ is the list of all components of $q_1$ and $l-s\ge l-6C-1\ge 6\max\{1,C\}+1$, hence, according
to part (c) of Lemma \ref{lem:regul} applied to $o_1$, there is $t \ge l-6\max\{1,C\}>s$ such that $p_t$ is connected to $p'_{t'}$ by
means of a path $r'_1$, where $s'+1 \le t' \le l'$, $(r'_1)_-=(p_t)_+$, $(r'_1)_+=(p'_{t'})_+$ and $\L(r_1') \le 1$. Consider the cycle
$o_2=r_1q_2r_1'q_2'$ in which $q_2$ and $q_2'$ are the segments of $q_1$ and $q_1'$ from $(p_s)_+=(r_1)_+$ to $(p_t)_+$ and
from $(p'_{t'})_+$ to $(p'_{s'})_+=(r_1)_-$ respectively (Fig. \ref{fig:1}).

Note that $p_{s+1},\dots, p_t$ is the list of all components of $q_2$ and $p'_{s'+1},\dots,p'_{t'}$ is the list of all components of
${q_2'}^{-1}$.  The cycle
$o_2$ satisfies the assumptions of part (a) of Lemma \ref{lem:regul}, therefore for every $i \in \{1,\dots,t-s\}$ there exists
$i' \in \{1,\dots,t'-s'\}$ such that $p_{s+i}$ is connected to $p'_{s'+i'}$ ($p_{s+i}$ can not be connected to $r_1$ [$r_1'$]
because in this case it would be connected to $p_s$ [$p_t$], but $q$ is without backtracking by Lemma \ref{lem:no_back}).

It remains to show that $i'=i$ for every such $i$. Indeed, if $i'<i$ for some $i \in \{1,\dots,t-s\}$
then one can consider the cycle $o_3=r_1q_3r_3'q_3'$, where
$q_3$ and $q_3'$ are segments of $q_2$ and $q_2'$ from $(q_2)_-=(r_1)_+$ to $(p_{s+i})_+$ and
from $(p'_{s'+i'})_+$ to $(q_2')_+=(r_1)_-$ respectively, and $(r_3')_-=(q_3)_+$, $(r_3')_+=(q_3')_-$, $\L(r_3') \le 1$.
According to part (a) of Lemma \ref{lem:regul}, each of the components
$p_{s+1},\dots, p_{s+i}$ of $q_3$ must be connected to one of $p'_{s'+1},\dots,p'_{s'+i'}$. Hence, since $i'<i$,
two distinct components of $q_3$ will be connected to the same component of ${q'_3}^{-1}$, which is impossible
by part (c) of Lemma \ref{lem:regul}.

The inequality $i' >i$ would lead to a contradiction after an application of a symmetric argument to $q_3'$.
Therefore $i'=i$ and the lemma is proved.
\end{proof}

\begin{lemma} \label{lem:conseq-reg} In the above notations, let $m \ge 7K$ and $C=\max\{ \L(r),\L(r')\}$.
For any positive integer $d$ there exists a constant $L=L(C,d) \in \N$ such that if $\L(q)\ge L$ then
there are $d$ consecutive components $p_s,\dots,p_{s+d-1}$ of $q$ and
$p'_{s'},\dots,p'_{s'+d-1}$ of $q'^{-1}$, so that $p_{s+i}$ is connected to $p'_{s'+i}$ for each $i=0,\dots,d-1$.
\end{lemma}

\begin{proof} Choose the constant $L$ so that $(L-1)/2 \ge 12\max\{C,1\}+2+d$.  
Let $p_1,\dots,p_l$ be the consecutive list all components of $q$.
Since $\lab(q)\in \cW(\Omega,m)$, we have  $l \ge (L-1)/2$ (due to the form of any word from $\cW(\Omega,m)$).
Thus we can apply Lemma \ref{lem:conseq-reg-strong} to find indices $s,t$ from its claim. 
By the choice of $s$ and $t$,  and the estimate on $l$, we have $t-s \ge d+1$, yielding the statement of the lemma.
\end{proof}

\begin{cor} \label{cor:hyp-elts} Let $G$ be a group hyperbolic relative to a family of proper subgroups $\Hl$.
Suppose that $a \in H_{\lambda_0}$, for some
${\lambda_0} \in \Lambda$, is an element of infinite order, and $x_1,x_2 \in G \setminus H_{\lambda_0}$.
Then there exists $k \in \N$ such that
$g=a^{k_1}x_1a^{k_2}x_2$ is a hyperbolic element of infinite order in $G$ whenever $|k_1|,|k_2| \ge k$.
\end{cor}

\begin{proof} Without loss of generality we can assume that $x_1,x_2 \in \cX$, since relative hyperbolicity
does not depend on the choice of the finite relative generating set (\cite[Thm. 2.34]{Osin-RHG}).
Choose the finite subset $\Omega \subset G$ and the constant $K\in \N$ according to the claim of Lemma \ref{lem:omega}, and set
$m=7 K$. As the order of $a$ is infinite, there is $ k \in \N$ such that
$a^{k'} \notin \{h \in \langle\Omega\rangle~:~|h|_\Omega \le m \}$ whenever $|k'| \ge k$.
Assume that $|k_1|,|k_2| \ge k$.

Suppose, first, that $g^l=1$ for some $l \in \N$. Consider the cycle $o=rqr'q'$ in {\ga} where $q_-=q_+=1$,
$\lab(q) \equiv (a^{k_1}x_1a^{k_2}x_2)^l \in \cW(\Omega,m)$
($a^{k_j} $ are considered as single letters from the alphabet $\cX \cup \cH$) and
$r,r',q'$ are trivial paths (consisting of a single point). Then, by part (a) of Lemma \ref{lem:regul},
every component of $q$ must be regular in $o$,
which is impossible since $q$ is without backtracking according to Lemma \ref{lem:no_back}. Hence $g$ has infinite order in $G$.

Suppose, now, that there exists $\lambda' \in \Lambda$, $u \in H_{\lambda'}$ and $y \in G$ such that $ygy^{-1}=u$.
Denote $C= |y|_{\cX \cup \cH}$.
Since element $u \in G$ has infinite order, there exists $l \in \N$ such that
$2l \ge 6C+2$ and $u^l \notin \{h \in \langle\Omega\rangle~:~|h|_\Omega \le m \}$.
The equality $yg^ly^{-1}u^{-l}=1$ gives rise to the cycle $o=rqr'q'$ in {\ga}, where $r$ and $r'$ are paths of length $C$ whose
labels represent $y$ in $G$, $r_-=1$, $q_-=r_+=y$, $\lab(q) \equiv (a^{k_1}x_1a^{k_2}x_2)^l \in \cW(\Omega,m)$,
$r'_-=q_+$, $q'_-=r'_+=y(a^{k_1}x_1a^{k_2}x_2)^ly^{-1}$
and $\lab(q') \equiv u^{-l} \in \cW(\Omega,m)$, $\L(q')=1$. 
By part (c) of Lemma \ref{lem:regul}, at least $2l-6C\ge 2$ distinct
components of $q$ must be connected to distinct components of $q'$, which is impossible as $q'$ has only one component.
The contradiction shows that $g$ must be a hyperbolic element of $G$.
\end{proof}

\begin{lemma}\label{lem:comm-spec} Let $G$ be a torsion-free group hyperbolic relative to a family of proper subgroups $\Hl$,
$a \in H_{\lambda_0} \setminus \{1\}$, for some $\lambda_0 \in \Lambda$, and $t,u \in G \setminus H_{\lambda_0}$.
Suppose that there exists $\hat k \in \N$ such that for every $k \ge \hat k$ the element $g_1=a^kta^kt^{-1}$ is commensurable with
$g_2=a^kua^ku^{-1}$ in $G$. Then there are $\beta,\gamma \in H_{\lambda_0}$ and
$\epsilon,\xi \in \{-1,1\}$ such that $u=\gamma t^\xi \beta$,
$\beta a \beta^{-1}=a^\epsilon$, $\gamma^{-1} a \gamma=a^\epsilon$.
\end{lemma}

\begin{proof} Changing  the finite relative generating set $\cX$ of $G$, if necessary,
we can assume that $t,u,t^{-1},u^{-1} \in \cX$. Let the finite subset $\Omega \subset G$ and the constant $K\in \N$ be chosen according
to Lemma \ref{lem:omega}. Define $m=7K$ and
suppose that $k$ is large enough to satisfy $a^k \notin \{h \in \langle\Omega\rangle~:~|h|_\Omega \le m\}$.

Since $g_1$ and $g_2$ are commensurable, there exist $l,l' \in \Z \setminus \{0\}$ and $y\in G$ such that
$yg_2^{l}y^{-1}=g_1^{l'}$. Let $C=|y|_{\cX\cup \cH}$, $d=8$ and $L=L(C,d)$ be the constant from Lemma \ref{lem:conseq-reg}.
Without loss of generality, assume that $4l\ge L$. Consider the cycle $o=rqr'q'$ in {\ga} such that  $r$ and $r'$ are paths of
length $C$ whose labels represent $y$ in $G$, $r_-=1$, $q_-=r_+=y$, $\lab(q) \equiv (a^kua^ku^{-1})^{l} \in \cW(\Omega,m)$, $\L(q)=4l$,
$r'_-=q_+$, $q'_-=r'_+=yg_2^l y^{-1}$, $\lab(q') \equiv (a^kta^kt^{-1})^{l'} \in \cW(\Omega,m)$, $\L(q')=4l'$.

Now, by Lemma \ref{lem:conseq-reg}, there are subpaths
$\tilde q=p_1s_1p_2s_2p_3s_3p_4$ of $q$ and
$\tilde q'=p'_1s'_1p'_2s'_2p'_3s'_3p_4'$ of $q'^{-1}$ such that
$\lab(p_i) \equiv a^k$, $\lab(p_i') \equiv a^{\epsilon k}$, $i=1,2,3,4$,
for some $\epsilon \in \{-1,1\}$ (which depends on the sign of $l'$), $\lab(s_1)\equiv\lab(s_3) \equiv u$,
$\lab(s_2) \equiv u^{-1}$, $\lab(s_1')\equiv\lab(s_3') \equiv t^\xi$,
$\lab(s_2') \equiv t^{-\xi}$, for some $\xi \in \{-1,1\}$, and $p_i$ is connected in {\ga} to
$p_i'$ for each $i=1,2,3,4$. Therefore there exist paths $\tilde p_1, \tilde p_2, \tilde p_3, \tilde p_4$ whose labels represent the elements
$\alpha,\beta,\gamma,\delta \in H_{\lambda_0}$ respectively, such that
$(\tilde p_1)_-=(p_1)_+$, $(\tilde p_1)_+=(p'_1)_+$, $(\tilde p_2)_-=(p_2')_+$, $(\tilde p_2)_+=(p_2)_+$,
$(\tilde p_3)_-=(p_3)_-$, $(\tilde p_3)_+=(p_3')_-$, $(\tilde p_4)_-=(p_4')_-$, $(\tilde p_4)_+=(p_4)_-$
(see Fig. \ref{fig:2}).

\begin{figure}[!ht]
  \begin{center}
   \input{figure2.pstex_t}

  \end{center}
  \caption{}\label{fig:2}
\end{figure}

The cycles $s_1^{-1}\tilde p_1 s_1' p_2' \tilde p_2 p_2^{-1}$,
$s_2 \tilde p_3 {s_2'}^{-1} \tilde p_2$ and $s_3^{-1} p_3^{-1} \tilde p_3 p'_3 s'_3 \tilde p_4$
 give rise to the following equalities in the group $G$:
$$u=\alpha t^\xi a^{\epsilon k} \beta a^{-k},~u=\gamma t^\xi \beta ~\mbox{ and }~u=a^{-k}\gamma a^{\epsilon k} t^\xi \delta.$$
Consequently, recalling that $H_{\lambda_0}$ is malnormal (Lemma \ref{lem:malnorm}) and that $t^\xi \notin H_{\lambda_0}$, we get
$$\beta a^k \beta^{-1} a^{-\epsilon k}=t^{-\xi}\gamma^{-1} \alpha t^{\xi} \in H_{\lambda_0} \cap t^{-\xi} H_{\lambda_0} t^\xi=\{1\},
~\mbox{ and} $$
$$a^{-\epsilon k} \gamma^{-1} a^k \gamma=t^\xi \delta \beta^{-1} t^{-\xi}  \in H_{\lambda_0} \cap t^\xi H_{\lambda_0} t^{-\xi}=\{1\}.$$

Thus \begin{equation}\label{eq:b-g-k} \beta a^k \beta^{-1}=a^{\epsilon k}~\mbox{ and }~  \gamma^{-1} a^k \gamma=a^{\epsilon k}
\end{equation} for some
$\beta=\beta(k), \gamma=\gamma(k) \in H_{\lambda_0}$ and $\epsilon=\epsilon(k),\xi=\xi(k) \in \{-1,1\}$.
Note that the proof works for any
sufficiently large $k$, therefore we can find two mutually prime positive integers $k,k'$ with the above properties
such that $\epsilon(k)=\epsilon(k')=\epsilon$ and $\xi(k)=\xi(k')=\xi$. Denote $\beta'=\beta(k')$ and $\gamma'=\gamma(k')$, then
$\gamma t^\xi \beta=u=\gamma' t^\xi \beta'$, implying
$$\gamma^{-1}\gamma'=t^{\xi}\beta {\beta'}^{-1} t^{-\xi} \in H_{\lambda_0} \cap t^\xi H_{\lambda_0} t^{-\xi}=\{1\}.$$ Hence
$\beta'=\beta$, $\gamma'=\gamma$,
\begin{equation} \label{eq:b-g-k'} \beta a^{k'} \beta^{-1}=a^{\epsilon k'} \mbox{ and }~ \gamma^{-1} a^{k'} \gamma=a^{\epsilon k'}.
\end{equation}

It remains to observe that since $k$ and $k'$ are mutually prime, the formulas \eqref{eq:b-g-k} and \eqref{eq:b-g-k'} together yield
$$\beta a \beta^{-1}=a^{\epsilon}~\mbox{ and }~  \gamma^{-1} a \gamma=a^{\epsilon},$$
q.e.d.
\end{proof}

\section{Small cancellation over relatively hyperbolic groups}\label{sec:smal_canc}
Let $G$ be a group generated by a subset $\mathcal{A} \subseteq G$ and let $\mathcal O$ be the set of all words in the alphabet
$\mathcal{A}^{\pm 1}$, that are trivial in $G$. Then $G$ has a presentation of the following form:
\begin{equation} \label{eq:G} G=\langle \mathcal{A}~\|~\mathcal{O}\rangle. \end{equation}
Given a symmetrized set of words $\mathcal{R}$
over the alphabet $\mathcal{A}$, consider the group $G_1$ defined by
\begin{equation}\label{eq:G_1} G_1=\langle \mathcal{A}~\|~\mathcal{O} \cup \mathcal{R}\rangle=
\langle G~\|~ \mathcal{R}\rangle.
\end{equation}

During the proof of the main result of this section we use presentations \eqref{eq:G_1}
(or, equivalently, the sets of additional relators $\mathcal R$) that satisfy
the {\it generalized small cancellation condition} $C_1(\varepsilon,\mu,\lambda,c,\rho)$. In the case of word hyperbolic groups
this condition was suggested by Ol'shanskii in \cite{Olsh2}, and was afterwards
generalized to relatively hyperbolic groups by Osin in \cite{Osin-SCT}. For the definition and detailed theory
we refer the reader to the paper \cite{Osin-SCT}, as we will only use the properties,
that were already established there. The following observation is an immediate consequence of the definition:

\begin{remark} \label{rem:C-1-C_1} Let the constants $\varepsilon_j,\mu_j,\lambda,c,\rho_j$, $j=1,2$, satisfy
$0 < \lambda \le 1$, $0 \le \varepsilon_1 \le \varepsilon_2$, $c \ge 0$, $0 < \mu_2 \le \mu_1$, $\rho_2 \ge \rho_1>0$.
If the presentation \eqref{eq:G_1} enjoys the condition $C_1(\varepsilon_2,\mu_2,\lambda,c,\rho_2)$ then it also enjoys the condition
$C_1(\varepsilon_1,\mu_1,\lambda,c,\rho_1)$.
\end{remark}

We will also assume that the reader is familiar with the notion of a {\it van Kampen diagram} over the group presentation \eqref{eq:G_1}
(see \cite[Ch. V]{L-S} or \cite[Ch. 4]{Olsh0}). Let $\Delta$ be such a diagram. A cell $\Pi$ of $\Delta$ is called an
$\mathcal R$-{\it cell} if the label of its boundary contour  $\partial \Pi$ (i.e., the word written on it starting with some
vertex in the counter-clockwise direction) belongs to $\mathcal R$.

Consider a simple closed path $o=rqr'q'$ in a diagram $\Delta$ over the presentation \eqref{eq:G_1}, such that $q$ is a
subpath of the boundary cycle of an $\mathcal R$-cell $\Pi$ and $q'$ is a subpath of $\partial \Delta$.
Let $\Gamma$ denote the subdiagram of $\Delta$ bounded by $o$. Assuming that $\Gamma$ has no holes, no
$\mathcal R$-cells and $\L(r),\L(r') \le \varepsilon$, it will be called an $\varepsilon$-{\it contiguity subdiagram}
of $\Pi$ to $\partial \Delta$. The ratio $\L(q)/\L(\partial \Pi)$ will be called the {\it contiguity degree} of $\Pi$ to $\partial \Delta$
and denoted $(\Pi,\Gamma,\partial \Delta)$.

A diagram is said to be {\it reduced} if it has a minimal number of $\mathcal R$-cells among all the diagrams
with the same boundary label.

If $G$ is a group hyperbolic relative to a family of proper subgroups $\{H_i\}_{i\in I}$, with a finite relative generating set $\cX$,
then $G$ is generated by the set $\mathcal{A}=\cX \cup \bigcup_{i \in I} (H_i \setminus \{1\})$, and the Cayley
graph $\Gamma(G,\mathcal{A})$ is a hyperbolic metric space \cite[Cor. 2.54]{Osin-RHG}.

As for every condition of small cancellation, the main statement of the theory is the following
analogue of Greendlinger's Lemma, claiming the existence of a cell, large part of whose contour lies on the boundary of
the van Kampen diagram.

\begin{lemma}[\cite{Osin-SCT}, Cor. 4.4] \label{lem:sm_canc_gamma-cell}
Suppose that the group $G$ is generated by a subset $\mathcal{A}$ such that the Cayley graph $\Gamma(G,\mathcal{A})$ is hyperbolic.
Then for any $0< \lambda \le 1$  there is $\mu_0 >0$ such that
for any $\mu \in (0,\mu_0]$ and $c \ge 0$ there are $\varepsilon_0 \ge 0$ and $\rho_0 >0$ with the following property.

Let the symmetrized presentation \eqref{eq:G_1} satisfy the $C_1(\varepsilon_0,\mu,\lambda,c,\rho_0)$-condition. Further,
let $\Delta$ be a reduced van Kampen diagram over $G_1$ whose boundary contour is $(\lambda,c)$-quasigeodesic in $G$.
 Then, provided $\Delta$ has an $\mathcal R$-cell,
there exists an $\mathcal R$-cell $\Pi$ in $\Delta$ and an $\varepsilon_0$-contiguity subdiagram
$\Gamma$ of $\Pi$ to $\partial \Delta$, such that
$$(\Pi,\Gamma,\partial \Delta) > 1-23\mu.$$
\end{lemma}

The main application of this particular small cancellation condition is
\begin{lemma}[\cite{Osin-SCT}, Lemmas 5.1 and 6.3] \label{lem:sm_canc_appl} For any $0 < \lambda \le 1$, $c \ge 0$ and $N>0$ there
exist $\mu_1 >0$, $\varepsilon_1 \ge 0$ and $\rho_1 >0$ such that for any symmetrized set of words $\mathcal{R}$ satisfying
$C_1(\varepsilon_1,\mu_1,\lambda,c,\rho_1)$-condition the following hold.

\begin{itemize}
\item[1.] The group $G_1$ defined by \eqref{eq:G_1} is hyperbolic relative to the collection of images $\{\eta(H_i)\}_{i\in I}$
under the natural homomorphism $\eta:G \to G_1$.
\item[2.] The restriction of $\eta$ to the subset of elements having length at most $N$ with
respect to $\mathcal A$ is injective.
\item[3.] Any element that has a finite order in $G_1$ is an image of an element of finite order in $G$.
\end{itemize}
\end{lemma}

Below is the principal lemma of this section that will later be used to prove Theorem \ref{thm:gp=out}.

\begin{lemma}\label{lem:add_rel} Assume that $G$ is a torsion-free group hyperbolic relative to a family of proper subgroups
$\{H_i\}_{i\in I}$, $\cX$ is a finite relative generating set of $G$, $S$ is a suitable subgroup of $G$ and $U \subset G$ is a finite subset.
Suppose that $i_0 \in I$, $a \in H_{i_0} \setminus \{1\}$ and $v_1,v_2 \in G$
are hyperbolic elements which are not commensurable to each other.
Then there exists a word $W(x,y)$ over the alphabet $\{x,y\}$ such that the
following is true.

Denote $w_1=W(a,v_1) \in G$, $w_2=W(a,v_2) \in G$, and let $\langle \langle w_2 \rangle \rangle$ be the normal closure of $w_2$ in $G$,
$G_1=G/\langle \langle w_2 \rangle \rangle$ and $\eta:G \to G_1$ be the natural epimorphism. Then
\begin{itemize}
\item $\eta$ is injective on $\Hl \cup U$ and $G_1$ is hyperbolic relative to the family $\{\eta(H_\lambda)\}_{\lambda \in \Lambda}$;
\item $\eta(S)$ is a suitable subgroup of $G_1$;
\item $G_1$ is torsion-free;
\item $\eta(w_1) \neq 1$.
\end{itemize}
\end{lemma}

\begin{proof} By Lemma \ref{lem:suit-inf_many_non-comm} there are hyperbolic elements $v_3,v_4 \in S$ such that
$v_i \stackrel{G}{\not\approx} v_j$ if $1 \le i<j \le 4$. Then by Lemma \ref{lem:Eg},
the group $G$ is hyperbolic relative to the finite collection of subgroups $\{H_i\}_{i\in I} \cup \bigcup_{j=1}^4 \{E_G(v_j)\}$,
and generated by the set $$\mathcal{A}=\cX \cup \left(\bigcup_{i \in I} H_i \cup \bigcup_{j=1}^4 E_G(v_j)\right)\setminus\{1\}.$$
Let $\Omega \subset G$ and $K\in \N$ denote the finite subset and the constant achieved after
an application of Lemma \ref{lem:omega} to this new collection of peripheral subgroups.

Define $m=7K$, $\lambda=1/3$, $c=2$ and $N=\max\{|u|_{\mathcal{A}}~:~u\in U\}+1$. Choose
$\mu_j>0$, $\varepsilon_j\ge 0$ and $\rho_j>0$, $j=0,1$, according to the claims of
Lemmas \ref{lem:sm_canc_gamma-cell} and \ref{lem:sm_canc_appl}.
Let $\varepsilon=\max\{\varepsilon_0,\varepsilon_1\}$, and let $L=L(C,d)>0$ be the constant given by Lemma \ref{lem:conseq-reg} where
$C=\varepsilon_0$ and $d=2$.
Evidently there exists $n \in \N$ such that, for $\mu=(3\varepsilon+11)/n$, one has
$$0<\mu \le \min\{\mu_0,\mu_1\},~ 2n(1-23 \mu)>L,~\mbox{ and }~ 2n>\max\{\rho_0,\rho_1\}.$$
Set
$$\mathcal{F}(\varepsilon)=\bigl\{ h \in \langle \Omega \rangle~:~|h| \le \max\{K(32\varepsilon+70),m\} \bigr\}.$$
Since the subset $\mathcal{F}(\varepsilon)$ is finite, we can find $k \in \N$ such that
$a^{k'},v_1^{k'}, v_2^{k'} \notin \mathcal{F}(\varepsilon)$ whenever $k' \ge k$.
Consider the word $$W(x,y)\equiv x^ky^kx^{k+1}y^{k+1} \dots x^{k+n-1} y^{k+n-1}.$$
Let $w_j\in G$ be the element represented by the word   $W(a,v_j)$ in $G$, $j=1,2$,
and let $\mathcal{R}$ be the set of all cyclic shifts of $W(a,v_2)$
and their inverses. By Lemma \ref{lem:malnorm}, $H_{i_0} \cap E_G(v_2)=\{1\}$ because $G$ is torsion-free, hence by
\cite[Thm. 7.5]{Osin-SCT} the presentation
\eqref{eq:G_1} satisfies the condition $C_1(\varepsilon,\mu,1/3,2,2n)$, and therefore,
by Remark \ref{rem:C-1-C_1}, it satisfies the conditions
 $C_1(\varepsilon_0,\mu,1/3,2,\rho_0)$ and $C_1(\varepsilon_1,\mu_1,1/3,2,\rho_1)$.

Observe that $w_1 \neq 1$ in $G$ because, otherwise, there would have existed a closed path $q$ in $\Gamma(G,\mathcal{A})$ labelled
by the word $W(a,v_1)$, and, by part (a) of Lemma \ref{lem:regul}, all components of $q$ would have been regular in
the cycle $o=rqr'q'$ (where $r,r',q'$ are trivial paths), which is obviously impossible.

Denote $G_1=G/\langle \langle w_2 \rangle\rangle$ and let $\eta: G \to G_1$ be the natural epimorphism.
Then, according to Lemma \ref{lem:sm_canc_appl}, the group $G_1$ is is torsion-free, hyperbolic relative to
$\{\eta(H_i)\}_{i\in I} \cup \bigcup_{j=1}^4 \{\eta(E_G(v_j))\}$ and $\eta$ is injective on the set
$\bigcup_{i\in I} H_i \cup \bigcup_{j=1}^4 E_G(v_j) \cup U$
(because the length in $\mathcal A$ of any element from this set is at most $N$).
Since any elementary group is word hyperbolic, $G_1$ is also
hyperbolic relative to $\{\eta(H_i)\}_{i\in I}$ (by Lemma \ref{lem:exhyp}) and $\eta(v_3),\eta(v_4) \in \eta(S)$
become hyperbolic elements of infinite order in $G_1$, that are not commensurable with each other (by Lemma \ref{lem:malnorm}).
Therefore $E_{G_1}(\eta(v_3)) \cap E_{G_2}(\eta(v_4))=\{1\}$ (recall that these
subgroups are cyclic by Lemma \ref{lem:Eg} and because $G_1$ is torsion-free), and, consequently, $\eta(S)$ is a suitable subgroup
of $G_1$.

Suppose that $\eta(w_1)=1$. By van Kampen's Lemma there exists a reduced planar diagram $\Delta$ over the presentation \eqref{eq:G_1}
with the word $W(a,v_1)$ written on its boundary. Since $W(a,v_1) \stackrel{G}{\neq} 1$, $\Delta$ possesses at least one $\mathcal{R}$-cell.
It was proved in \cite[Lemma 7.1]{Osin-SCT} that any path in $\Gamma(G,\mathcal{A})$ labelled by $W(a,v_1)$ is $(1/3,2)$-quasigeodesic,
hence we can apply Lemma \ref{lem:sm_canc_gamma-cell} to find an $\mathcal{R}$-cell $\Pi$ of $\Delta$ and an $\varepsilon_0$-contiguity
subdiagram $\Gamma$ (containing no $\mathcal R$-cells) between $\Pi$ and $\partial \Delta$ such that
$(\Pi,\Gamma,\partial \Delta)>1-23\mu$. Thus there exists a cycle $o=rqr'q'$ in $\Gamma(G,\mathcal{A})$ such that
$q$ is labelled by a subword of (a cyclic shift of) $W(a,v_2)$, $q'$ is labelled by a subword of (a cyclic shift of) $W(a,v_1)^{\pm 1}$,
$\L(r),\L(r') \le \varepsilon_0=C$ and $$\L(q) > (1-23 \mu) \cdot \L(\partial \Pi) = (1-23 \mu)\cdot 2n> L.$$
In particular, $\lab(q),\lab(q') \in \cW(\Omega,m)$. Therefore we can apply Lemma \ref{lem:conseq-reg} to find two consecutive components
of $q$ that are connected to some components of $q'$. Due to the form of the word $W(a,v_2)$, one of the formers will have to be an
$E_G(v_2)$-component, but $q'$ can have only $E_G(v_1)$- or $H_{i_0}$-components. This yields a contradiction because $E_G(v_2)\neq E_G(v_1)$
and $E_G(v_2) \neq H_{i_0}$. Hence $\eta(w_1) \neq 1$ in $G_1$, and the proof is complete.
\end{proof}

\section{Every group is a group of outer automorphisms of a (2CC)-group}\label{sec:every_gp=out}
\begin{lemma} \label{lem:hyp-spec-gen} There exists a word $R(x,y)$ over the two-letter alphabet $\{x,y\}$ such that
every non-elementary torsion-free word hyperbolic group $F_1$ has a non-elemen\-tary torsion-free word hyperbolic
quotient $F$ that is generated by two elements $a,b \in F$ satisfying
\begin{equation} \label{eq:Rab} R(a,b) \stackrel{F}{\neq} 1,~
R(a^{-1},b^{-1})\stackrel{F}{=} 1, ~R(b,a)\stackrel{F}{=} 1, ~R(b^{-1},a^{-1}) \stackrel{F}{=} 1.\end{equation}
\end{lemma}

\begin{proof} Consider the word $$R(x,y) \equiv xy^{101} x^2 y^{102} \dots x^{100} y^{200}.$$
Denote by $F(a,b)$ the free group with the free generators $a,b$. Let
$$\mathcal{R}_1=\{R(a,b),R(a^{-1},b^{-1}), R(b,a),R(b^{-1},a^{-1})\},$$ and $\mathcal{R}_2$ be the set of all cyclic permutations of
words from $\mathcal{R}_1^{\pm 1}$. It is easy to see that the set $\mathcal{R}_2$ satisfies the classical small cancellation condition
$C'(1/8)$ (see \cite[Ch. V]{L-S}). Denote by $\tilde N$ the normal closure of the set
$$\mathcal{R}_3=\{R(a^{-1},b^{-1}), R(b,a),R(b^{-1},a^{-1})\}$$ in
$F(a,b)$. Since the symmetrization of $\mathcal{R}_3$ also satisfies $C'(1/8)$,
the group $\tilde F=F(a,b)/{\tilde N}$ is a torsion-free (\cite[Thm. V.10.1]{L-S})
word hyperbolic group (because it has a finite presentation for which the Dehn function is linear by \cite[Thm. V.4.4]{L-S}) such that
$$R(a,b) \stackrel{\tilde F}{\neq} 1~\mbox{ but }~
R(a^{-1},b^{-1})\stackrel{\tilde F}{=}R(b,a)\stackrel{\tilde F}{=}R(b^{-1},a^{-1}) \stackrel{\tilde F}{=} 1.$$
Indeed, if the word $R(a,b)$ were trivial in $\tilde F$ then,
by Greendlinger's Lemma \cite[Thm. V.4.4]{L-S}, it would contain more than a half of a relator from (the symmetrization of)
$\mathcal{R}_3$ as a subword, which would contradict the fact that $\mathcal{R}_2$ enjoys $C'(1/8)$.
The group $\tilde F$ is non-elementary because every torsion-free  elementary group is cyclic, hence, abelian, but in any
abelian group the relation $R(a^{-1},b^{-1})=1$ implies $R(a,b) = 1$.

Now, the free product $\tilde G=\tilde F * F_1$ is a torsion-free hyperbolic group. Its subgroups $\tilde F$ and $F_1$ are non-elementary,
hence, according to a theorem of Ol'shanskii \cite[Thm. 2]{Olsh2}, there exists a non-elementary torsion-free word hyperbolic
group $F$ and a homomorphism $\phi: \tilde G \to F$ such that $\phi(\tilde F)=\phi(F_1)=F$ and $\phi(R(a,b)) \neq 1$ in $F$. Therefore
$F$ is a quotient of $F_1$, the ($\phi$-images of the) elements $a,b$ generate $F$ and enjoy the required relations.
\end{proof}

We are now ready to prove Theorem \ref{thm:gp=out}.

\begin{proof}[Proof of Theorem \ref{thm:gp=out}.]  The argument will be similar to the one used to prove Theorem \ref{thm:ext-main}.

First, set $n=2$ and apply Lemma \ref{lem:ncc-normal} to find a countable torsion-free group $H$ and a normal subgroup
$M \lhd H$, where $H/M \cong C$ and $M$ has {\cc} (alternatively, one could start with a free group $H'$ and $M' \lhd H'$
such that $H'/M' \cong C$, and then apply Lemma \ref{lem:emb-non-simple} to the pair $(H',M')$ to obtain $H$ and $M$ with these
properties). Consider the word $R(x,y)$ and the torsion-free hyperbolic group $F$,
generated by the elements $a,b \in F$ which satisfy \eqref{eq:Rab}, given by
Lemma \ref{lem:hyp-spec-gen}. Denote $G(-2)=H*F$ and let $N(-2)$ be the normal closure of $\langle M, F \rangle$ in $G(-2)$, $F(-2)=F$,
${\mathfrak R}(-2)=\{R(a,b)\}$ -- a finite subset of $F(-2)$. By Lemma \ref{lem:hyp_suit_sbgp_free_prod},
$G(-2)$ will be hyperbolic relative to the
subgroup $H$, $G(-2)=H\cdot N(-2)$, $H \cap N(-2)=M$ and $F(-2)$ will be a suitable subgroup of $G(-2)$.

The element $a \in F(-2)$ will be hyperbolic in $G(-2)$ and since the group $G(-2)$ is torsion-free, the maximal elementary subgroup
$E_{G(-2)}(a)$ will be cyclic generated by some element $h_{-2}x_{-2}$, where $h_{-2} \in H$, $x_{-2} \in N(-2)$.

Choose $y_{-2} \in M$ so that $h_{-2}y_{-2} \neq 1$.
By Lemmas \ref{lem:Eg} and \ref{lem:HNN-rel_hyp}, the HNN-extension
$$G(-3/2)=\langle G(-2),t_{-1}~\|~t_{-1}h_{-2}x_{-2}t_{-1}^{-1}=h_{-2}y_{-2}\rangle$$
is hyperbolic
relative to $H$. As in proof of Theorem \ref{thm:ext-main}, one can verify that $F(-3)$ is a suitable subgroup of $G(-3/2)$, and apply
Theorem \ref{thm:main_SCT} to find an epimorphism $\eta_{-2}:G(-3/2) \to G(-1)$ such that $G(-1)$ is a torsion-free group
hyperbolic relative to $\eta_{-2}(H)$, $\eta_{-2}$ is injective on $H \cup {\mathfrak R}(-2)$ and $\eta_{-2}(t_{-1}) \in F(-1)$ where
$F(-1)=\eta_{-2}(F(-2))$ is a suitable subgroup of $G(-1)$. Hence $\eta_{-2}(G(-2))=G(-1)$ as $G(-3/2)$ was generated by $G(-2)$ and $t_{-1}$.

Denote $N(-1)=\eta_{-2}(N(-2))$,  ${\mathfrak R}(-1)=\eta_{-2}({\mathfrak R}(-2))$ and
$\psi_{-2}=\left.{\eta_{-2}}\right|_{G(-2)}:G(-2) \twoheadrightarrow G(-1)$. One can show that $G(-1)=H\cdot N(-1)$ and $H \cap N(-1)=M$
using the same arguments as in the proof of Theorem \ref{thm:ext-main}.
According to the construction, we have
$$\eta_{-2}(t_{-1})\eta_{-2}(a) \eta_{-2}(t_{-1}^{-1}) =\eta(t_{-1}at_{-1}^{-1}) \in N(-1) \cap H = M$$ in $G(-1)$, therefore, since the conjugation by
$\eta_{-2}(t_{-1})$ is an inner automorphism of $F(-1)$, we can assume that $F(-1)$ is generated by $a_{-1}$ and $b_{-1}$,
where $a_{-1} \in M$ and $R(a_{-1},b_{-1}) \neq 1$ in $F(-1)$ (because $\eta_{-2}(R(a,b)) \neq 1$ in $F(-1)$).

Now, if $b_{-1}$ is not a hyperbolic element of $G(-1)$, i.e., if $b_{-1} \stackrel{G(-1)}{\sim} c$ for some $c \in H$,
then $c \in N(-1) \cap H=M$, and since $M$ has {\cc} we can find $s_{-1} \in G(-1)$ such that $b_{-1}=s_{-1}a_{-1} s_{-1}^{-1}$.
In this case we define $G(0)=G(-1)$, $N(0)=N(-1)$, $F(0)=F(-1)$, ${\mathfrak R}(0)={\mathfrak R}(-1)$, $a_0=a_{-1}$, $s_0=s_{-1}$ and $\psi_{-1}=id_{G(-1)}$.

Otherwise, if $b_{-1}$ is hyperbolic in $G(-1)$, then we construct the group $G(0)$, and an epimorphism $\psi_{-1}:G(-1) \to G(0)$
in an analogous way, to make sure that $\eta_{-1}$ is injective on $H \cup {\mathfrak R}(-1)$, $G(0)$ torsion-free and
hyperbolic relative to (the image of) $H$, $F(0)=\psi_{-1}(F(-1))$ is a suitable subgroup of $G(0)$, $G(0)=H \cdot N(0)$ and
$H \cap N(0)=M$ where $N(0)=\psi_{-1}(N(-1))$, and $b_{0}=s_{0}a_{0} s_{0}^{-1}$ in $G(0)$ where
$b_0=\psi_{-1}(b_{-1})$, $a_0=\psi_{-1}(a_{-1})$ for some $s_0 \in G(0)$

Enumerate all elements of $N(0)$: $\{g_0,g_1,g_2,\dots\}$, and
of $G(0)$: $\{q_0,q_1,q_2,\dots\}$, so that $g_0=q_0=1$.

The groups $G(j)$ together with $N(j) \lhd G(j)$, $F(j) \le G(j)$, finite subsets ${\mathfrak R}(j) \subset G(j)$,
and elements $a_j,s_j \in G(j)$,
$j=1,2,\dots$, that we will construct shall satisfy the following properties:

\begin{itemize}
\item[$1^\circ$.] for each $j\in \N$ there is an epimorphism $\psi_{j-1}: G(j-1) \to G(j)$ which is injective on
$H\cup {\mathfrak R}(j-1)$. $F(j)=\psi_{j-1}(F(j-1))$, $N(j)=\psi_{j-1}(N(j-1))$, $a_j=\psi_{j-1}(a_{j-1}) \in M$,
$s_j=\psi_{j-1}(s_{j-1}) \in G(j)$;
\item[$2^\circ$.] $G(j)$ is torsion-free and hyperbolic relative to (the image of) $H$, and $F(j) \le G(j)$
is a suitable subgroup generated by $a_j$ and $s_ja_js_j^{-1}$;
\item[$3^\circ$.] $G(j)=H\cdot N(j)$, $N(j)\lhd G(j)$ and $H \cap N(j)=M$;
\item[$4^\circ$.] the natural image $\bar g_j$ of $g_j$  in $G(j)$ belongs to $F(j)$;
\item[$5^\circ$.] there exists $z_j \in H$ such that $\bar q_j \stackrel{G(j)}{\sim} z_j$, where $\bar q_j$
is the image of $q_j$ in $G(j)$;
\item[$6^\circ$.] if $j \ge 1$, $\bar q_{j-1} \in G(j-1)\setminus H$ and for each $\hat k \in \N$ there is $k \ge \hat k$ such that
$a_{j-1}^ks_{j-1}a_{j-1}^ks_{j-1}^{-1} \stackrel{G(j-1)}{\not \approx} a_{j-1}^k \bar q_{j-1}a_{j-1}^k \bar q_{j-1}^{-1}$,
then 
there is a word $R_{j-1}(x,y)$ over the two-letter alphabet $\{x,y\}$ which satisfies
$${\mathfrak R}(j) \ni \psi_{j-1}\left(R_{j-1}(a_{j-1},s_{j-1}a_{j-1}s_{j-1}^{-1}) \right) \neq 1~\mbox{ and }~$$
$$\psi_{j-1}\left(R_{j-1}(a_{j-1},\bar q_{j-1}a_{j-1}\bar q_{j-1}^{-1})\right) {=} 1~\mbox{ in } G(j).$$
\end{itemize}

Suppose that the groups $G(0),\dots, G(i)$ have already been defined. The group $G(i+1)$ will be constructed in three steps.

First, assume that $\bar q_i \in G(i) \setminus H$ and for each $\hat k \in \N$ there is $k \ge \hat k$ such that
$a_{i}^ks_{i}a_{i}^ks_{i}^{-1} \stackrel{G(i)}{\not \approx} a_{i}^k \bar q_{i}a_{i}^k \bar q_{i}^{-1}$. Observe that $s_i \notin H$
because, otherwise, one would have  $F(i) \subset H$, which is impossible as $F(i)$ is suitable in $G(i)$. Therefore, by Corollary
\ref{cor:hyp-elts}, we can suppose that $k$ is so large that the elements $v_1=a_{i}^ks_{i}a_{i}^ks_{i}^{-1}$ and
$v_2=a_{i}^k \bar q_{i}a_{i}^k \bar q_{i}^{-1}$ are hyperbolic in $G(i)$. Applying Lemma \ref{lem:add_rel} we can find a word
$W(x,y)$ over $\{x,y\}$ such that the group $G(i+1/3)=G(i)/\langle \langle W(a_i,v_2) \rangle\rangle$ and the
natural epimorphism $\eta: G(i) \to G(i+1/3)$ satisfy the following: $\eta$ is injective on $H \cup \mathfrak{R}(i)$,
$G(i+1/3)$ is torsion-free and hyperbolic relative to (the image of) $H$, $\eta(F(i)) \le G(i+1/3)$ is a suitable subgroup,
and $\eta(W(a_i,v_1)) \neq 1$. Define the word $R_{i}(x,y)\equiv W(x,x^ky^k)$. Then
$R_{i}(a_i,s_ia_is_i^{-1})=W(a_i,v_1)$, $R_i(a_i,\bar q_ia_i \bar q_i^{-1})=W(a_i,v_2)$ in $G(i)$, hence
$$\eta\left(R_{i}(a_i,s_ia_is_i^{-1})\right) \neq 1~\mbox{ and }~ \eta\left(R_i(a_i,\bar q_ia_i \bar q_i^{-1})\right)=1 ~\mbox{ in } G(i+1/3).$$

If, on the other hand, $\bar q_i \in H$ or there is $\hat k \in \N$ such that for every $k \ge \hat k$ one has
$a_{i}^ks_{i}a_{i}^ks_{i}^{-1} \stackrel{G(i)}{\approx} a_{i}^k \bar q_{i}a_{i}^k \bar q_{i}^{-1}$, then we define $G(i+1/3)=G(i)$,
$\eta: G(i) \to G(i+1/3)$ to be the identical homomorphism and $R_i(x,y)$ to be the empty word.

Let $\hat g_{i+1}$ and $\hat q_{i+1}$ denote the images of $g_{i+1}$ and $q_{i+1}$ in $G(i+1/3)$, $\hat N(i)=\eta(N(i))$, $\hat F(i)=\eta(F(i))$
and $\hat {\mathfrak{R}}(i)=\eta\left(\mathfrak{R}(i) \cup \{R_{i}(a_i,s_ia_is_i^{-1})\}\right)$. Then, using $3^\circ$,
we get  $G(i+1/3)=H\cdot \hat N(i)$ and $H \cap \hat N(i)=M$ because $\ker(\eta)\le N(i)$ (as $a_i, \bar q_ia_i \bar q_i^{-1} \in N(i)$).

Now we construct the group $G(i+2/3)$ in exactly the same way as the group $G(i+1/2)$ was constructed in during the
proof of Theorem \ref{thm:ext-main}.

If for some $f \in G(i+1/3)$, $f\hat q_{i+1} f^{-1}=z \in H$, then
set $G(i+2/3)=G(i)$, $K_{i+1}=\hat N(i)\lhd G(i+2/3)$ and $t_{i+1}=1$.

Otherwise, $\hat q_{i+1}$ is a hyperbolic element of infinite order in $G(i+1/3)$. Since $G(i+1/3)$ is torsion-free, one has
$E_{G(i+1/3)}(\hat q_{i+1}) = \langle h x \rangle$ for some $h \in H$ and $x \in \hat N(i)$, and there is $m \in \Z$ such
that $\hat q_{i+1}=(hx)^m$.
Now, by Lemma \ref{lem:Eg}, $G(i+1/3)$ is hyperbolic relative to $\{H,\langle hx \rangle\}$. Choose $y \in M$ so that
$hy \neq 1$ and let $G(i+2/3)$ be the following HNN-extension of $G(i+1/3)$:
$$G(i+2/3) = \langle G(i+1/3),t_{i+1}~\|~t_{i+1} (hx) t_{i+1}^{-1} = hy\rangle. $$
The group $G(i+2/3)$ is torsion-free and hyperbolic relative to $H$ by Lemma \ref{lem:HNN-rel_hyp}. One can show that $\hat F(i)$
is a suitable subgroup of $G(i+2/3)$ in the same way as during the proof of Theorem \ref{thm:ext-main}.
Lemma \ref{lem:HNN-univer} assures that $H \cap K_{i+1}=M$ where $K_{i+1} \lhd G(i+2/3)$ is the normal closure of
$\langle \hat N(i),t_{i+1} \rangle$ in $G(i+2/3)$. Finally,  note that
$$t_{i+1} \hat q_{i+1} t_{i+1}^{-1}=t_{i+1} (hx)^m t_{i+1}^{-1}=(hy)^m = z \in H~\mbox{ in } G(i+2/3).$$

Define $T_{i+1}=\{\hat g_{i+1}, t_{i+1}\}\subset K_{i+1}$. The group $G(i+1)$ is constructed from $G(i+2/3)$ as follows.
Since $T_{i+1} \cdot \hat F(i) \subset K_{i+1} \lhd G(i+2/3)$,
we can apply Theorem \ref{thm:main_SCT} to find a group $G(i+1)$ and an epimorphism $\varphi_i:G(i+2/3) \to G(i+1)$ such that
$\varphi_i$ is injective on $H \cup \hat{\mathfrak{R}}(i)$, $G(i+1)$ is torsion-free and hyperbolic relative to (the image of) $H$,
$\{\varphi_i(\hat g_{i+1}), \varphi_i(t_{i+1})\} \subset \varphi_i(\hat F(i))$, $\varphi_i(\hat F(i))$ is a suitable subgroup of $G(i+1)$,
and $\ker(\varphi_i) \le K_{i+1}$. Denote by $\psi_i:G(i) \to G(i+1)$ the
composition $\varphi_i \circ \eta$. Then $\psi_i(G(i))=\varphi_i(G(i))=G(i+1)$ because $G(i+2/3)$ was generated by $G(i)$ and
$t_{i+1}$, and according to the construction,
$t_{i+1} \in \varphi_i(\hat F(i))\le \varphi_i(G(i))$. Now, after defining $F(i+1)=\psi_{i}(F(i))$, $N(i+1)=\psi_{i}(N(i))$,
$\mathfrak{R}(i+1)=\varphi_i(\hat{\mathfrak{R}}(i))$,
$\bar g_{i+1}=\varphi_i(\hat g_{i+1}) \in F(i+1)$ and $z_{i+1}=\varphi_i(z) \in H$, we see that the conditions
$1^\circ$ - $5^\circ$ hold in the case when $j=i+1$, as in the proof of Theorem \ref{thm:ext-main}.
The last property $6^\circ$ follows from the way we constructed the group $G(i+1/3)$.

Let $Q=G(\infty)$ be the direct limit of the sequence $(G(i),\psi_i)$ as $i \to \infty$, and let $F(\infty)$ and $N=N(\infty)$ be the limits
of the corresponding subgroups. Let $a_\infty$, $b_\infty$ and $s_\infty$ be the images of $a_0$, $b_0$ and $s_0$ in $Q$ respectively.
Then $b_\infty=s_\infty a_\infty s_\infty^{-1}$, $Q$ is torsion-free by $2^\circ$,
$N\lhd Q$, $Q=H \cdot N$ and $H \cap N=M$ by $3^\circ$, $N \le F(\infty)$ by $4^\circ$. Hence $Q/N \cong H/M \cong C$.

Since $F(0) \le N(0)$ we get $F(\infty) \le N$.
Thus $N=F(\infty)$ is a homomorphic image of $F(0)=F$, and, consequently, it is a quotient of $F_1$.
By $5^\circ$, for any $q \in N$ there are $z \in H$ and $p \in Q$ such that $pqp^{-1}=z$. Consequently $z \in H\cap N=M$.
Choose $x \in N$ and $h \in H$ so that
$p=hx$. Since $M$ has {\cc} and $h^{-1}zh \in M$,  there is $y \in M$ such that $yh^{-1}zhy^{-1}=z$, therefore
$(yx) q (yx)^{-1}=z \in M$ and $yx \in MN =N$. Hence each element $q$ of $N$ will be conjugated (in $N$) to an element
of $M$, and since $M$ has {\cc}, therefore the group $N$ will also have {\cc}.

The property that $C_Q(N)=\{1\}$ can be established in the same way as in Theorem \ref{thm:ext-main}.
Therefore the natural homomorphism $Q \to Aut(N)$ is injective. It remains to show that it is surjective, that is for every
$\phi \in Aut(N)$ there is $g \in Q$ such that $\phi(x)=gxg^{-1}$ for every $x\in N$. Since all non-trivial elements of $N$
are conjugated, after composing $\phi$ with an inner automorphism of $N$, we can assume that $\phi(a_\infty)=a_\infty$.
On the other hand, there exist $q_\infty \in N$ and $i \in \N$ such that $\phi(b_\infty)=q_\infty a_\infty q_\infty^{-1}$ and
$q_\infty$ is the image of $q_i$ in $Q$. Note that $s_\infty \notin H$ because $s_i \in G(i)\setminus H$ for every $i \in \N$. This implies
that $H$ is a proper subgroup of $N$, thus $q_\infty \notin H$ since
$N=F(\infty)= \langle a_\infty, q_\infty a_\infty q_\infty^{-1}\rangle \le Q$, and $a_\infty \in H$. Hence $\bar q_i \in G(i)\setminus H$.

Now we have to consider two possibilities.

Case 1: for each $\hat k \in \N$ there is $k \ge \hat k$ such that
$$a_{i}^ks_{i}a_{i}^ks_{i}^{-1} \stackrel{G(i)}{\not \approx} a_{i}^k \bar q_{i}a_{i}^k \bar q_{i}^{-1}.$$
Then there is a word $R_i(x,y)$ such that the property $6^\circ$ holds for $j=i+1$. And, since
each $\psi_j$ is injective on $\{1\} \cup \mathfrak{R}_j$ (by $2^\circ$), we conclude that
$$R_i(a_\infty,s_\infty a_\infty s_\infty^{-1})\neq 1~\mbox{ and }~
R_i(a_\infty,q_\infty a_\infty q_\infty^{-1})=1~\mbox{ in } Q,$$
which contradicts the injectivity of $\phi$. Hence Case 1 is impossible.

Case 2: the assumptions of Case 1 fail. Then we can use Lemma \ref{lem:comm-spec} to find $\beta,\gamma \in H$ and
$\epsilon,\xi \in \{-1,1\}$ such that $\bar q_i=\gamma s_i^\xi \beta$, $\beta a_i \beta^{-1} =a_i^\epsilon$ and
$\gamma^{-1} a_i \gamma =a_i^\epsilon$ in $G(i)$. Denote by $\gamma_\infty$ the image $\gamma$ in $Q$, and for any $y \in Q$ let
$C_{y}$ be the automorphism of $N$ defined by
$C_{y}(x)=y x y^{-1}$ for all $x \in N$.

If $\xi=-1$ then $\gamma_\infty^{-1} a_\infty \gamma_\infty = a_\infty^\epsilon$ and
$\phi(b_\infty)=q_\infty a_\infty q_\infty^{-1}=\gamma_\infty s_\infty^{-1} a_\infty^\epsilon s_\infty \gamma_\infty^{-1}$,
hence $$Aut(N) \ni C_{s_\infty \gamma_\infty^{-1}} \circ \phi: \left\{ \begin{array}{rcl}
a_\infty  & \mapsto & s_\infty a_\infty^\epsilon s_\infty^{-1}=b_\infty^\epsilon \\
b_\infty=s_\infty a_\infty s_\infty^{-1} & \mapsto & a_\infty^\epsilon
\end{array} \right. .$$
But $N$ has no such automorphisms because $R(a_\infty,b_\infty) \neq 1$ and $R(b_\infty^\epsilon,a_\infty^\epsilon) = 1$
in $N$ (since $N$ is a quotient of $F$ and $1\neq R(a_0,b_0) \in \mathfrak{R}(0)$ in $G(0)$).

Therefore $\xi =1$. Similarly, $\epsilon=1$, as otherwise we would obtain a contradiction with the fact that
$R(a_\infty^{-1},b_\infty^{-1})=1$ in $N$. Thus
$$Aut(N) \ni C_{\gamma_\infty^{-1}} \circ \phi: \left\{ \begin{array}{rcl}
a_\infty  & \mapsto &  a_\infty \\
b_\infty=s_\infty a_\infty s_\infty^{-1} & \mapsto & s_\infty a_\infty s_\infty^{-1}=b_\infty
\end{array} \right. .$$
And since $a_\infty$ and $b_\infty$ generate $N$ we conclude that for all $x \in N$,  $\phi(x)= g x g^{-1}$, where
$g=\gamma_\infty \in Q$. Thus the natural homomorphism from $Q$ to $Aut(N)$ is bijective, implying that
$Out(N)=Aut(N)/Inn(N) \cong Q/N \cong C$. Q.e.d.
\end{proof}

\end{document}